\documentclass[review]{siamart1116}

\pdfoutput=1

\usepackage{color}
\usepackage{graphicx}              % to include figures
\usepackage{epstopdf}
\usepackage{amsmath}               % great math stuff
\usepackage{amssymb}
\usepackage{amsfonts}              % for blackboard bold, etc
\usepackage{comment}
\usepackage{lipsum}
\usepackage{algorithm,algpseudocode}    % for algorithm
\usepackage{enumerate}

\usepackage{mathtools}
\usepackage{caption}
\usepackage{subcaption}
\usepackage{longtable}
\usepackage{url} % for bibtex citing online articles
\usepackage{footnote}
\usepackage{lscape} % for wide table rotation
\usepackage{threeparttablex}

\newsiamthm{assumption}{Assumption}
\newsiamthm{remark}{Remark}
\newsiamthm{declaration}{Declaration}
\newsiamthm{observation}{Observation}

\makeatletter

\newcommand{\Rmnum}[1]{\expandafter\@slowromancap\romannumeral #1@}
\makeatother
\newcommand{\W}{\mathcal{W}}
\numberwithin{theorem}{section}

\newcommand{\TheTitle}{Decomposition Algorithms for Distributionally Robust Optimization using Wasserstein Metric} 
\newcommand{\TheAuthors}{Fengqiao Luo and Sanjay Mehrotra}

\headers{Decomposition Algorithm for Wasserstein-Robust Optimization}{\TheAuthors}

\title{{\TheTitle}\thanks{Submitted to the editors DATE.
\funding{This research was partially supported by NSF grants CMMI-1362003
and CMMI-1100868. 
}}}

\author{
  Fengqiao Luo\thanks{Department of Industrial Engineering and Management Science, Northwestern University, Evanston, IL, 60208
    (\email{fengqiaoluo2014@u.northwestern.edu}).}
  \and
  Sanjay Mehrotra\thanks{Department of Industrial Engineering and Management Science, Northwestern University, Evanston, IL, 60208 (\email{mehrotra@iems.northwestern.edu}).}
 }

\usepackage{amsopn}

\ifpdf
\hypersetup{
  pdftitle={\TheTitle},
  pdfauthor={\TheAuthors}
}
\fi

\begin{document}

\maketitle

\begin{abstract} 
We study distributionally robust optimization (DRO) problems where the ambiguity set is defined using the Wasserstein metric. 
We show that this class of DRO problems can be reformulated as semi-infinite programs.    
We give an exchange method to solve the reformulated problem 
for the general nonlinear model, and a central cutting-surface method for the convex case, assuming that we have a separation oracle. 
We used a distributionally robust generalization of the logistic regression model to test our algorithm. 
Numerical experiments on the distributionally robust logistic regression models show that the number of oracle calls are typically $20\sim 50$ to achieve 5-digit precision. 
The solution found by the model is generally better in its ability to predict with a smaller standard error.
\end{abstract}

% REQUIRED

% REQUIRED

\section{Introduction}
\label{sec:introduction}
We consider a distributionally robust optimization problem defined using the Wasserstein metric:
\begin{equation}
\label{opt:DRO}
\underset{\theta\in\Theta}{\textrm{inf}}\;\underset{P\in\mathcal{P}}{\textrm{sup}}\;\mathbb{E}_P[h(\theta,\xi)], \tag{WRO}
\end{equation}
where $\theta$ are the decision variables, and $\xi$ are the model parameters.
We assume that $h(\theta,\xi)$ is a continuous, but possibly a non-convex function in $\theta$.
Additional assumptions are made as results are developed in Sections~\ref{sec:reform} and \ref{sec:alg-SDP}.
Let $\xi$ be a random vector defined on a measurable space $(\Xi,\mathcal{F})$,
where $\Xi$ is the support in $\mathbb{R}^n$, and $\mathcal{F}$ is a $\sigma$-algebra that contains
all singletons in $\Xi$, i.e., $\{\xi\}\in\mathcal{F},\;\;\forall \xi\in\Xi$.  
We let the distribution $\mathcal{P}$ followed by $\xi$ be unknown, but assume that it belongs to an ambiguity 
set $\mathcal{P}$. Let $\mathcal{M}(\Xi,\mathcal{F})$ be the set of all probability distributions on $(\Xi,\mathcal{F})$,
and define $\mathcal{P}$ by using the $L_1$-Wasserstein metric:
\begin{equation}\label{def:ambiguity_set}
\mathcal{P}:=\big \{  P\in \mathcal{M}(\Xi,\mathcal{F})\; \big| \; \W(P, P_0) \le r_0 \big\},
\end{equation}
where $P_0$ is an empirical distribution satisfying: $P_0(\xi^i)=1/m,\;i\in[m]$, 
and $\{\xi^i\in\Xi:\;i\in[m]\}$ are the observed samples of $\xi$, 
and $r_0>0$ is a scalar.  The $L_1$-Wasserstein metric is defined as \cite{givens1984}:
\begin{equation} \label{def:Kantorovich_metric}
\W(P_1, P_2):= \underset{K\in\mathcal{S}(P_1,P_2)}{\textrm{inf}} \int_{\Xi\times \Xi} d\big( s_1, s_2 \big) K(ds_1\times d s_2), 
\end{equation}
where $\mathcal{S}(P_1,P_2):=\big\{ K\in\mathcal{M}(\Xi\times\Xi,\mathcal{F}\times\mathcal{F}):\; K(A\times\Xi) = P_1(A),\;  K(\Xi\times A) = P_2(A),\; \forall A\in\mathcal{F}   \big\}$ is the set of all joint probability distributions whose marginal distributions are $P_1$ and $P_2$, 
and $d(\cdot,\cdot)$ is a metric defined on $\Xi$, which is measurable in $\mathcal{F}\times\mathcal{F}$. 
The set $\Xi$ is partitioned as: $\Xi=\cup^{m+1}_{i=1}\Xi^i$, 
where $\Xi^i=\{ \xi^i\}$ for $i\in[m]$ and $\Xi^{m+1}=\Xi\setminus(\cup_{i\in[m]}\Xi^i)$.  
We call (\ref{opt:DRO}) a Wasserstein-robust optimization problem from hereafter.

Wasserstein metric has been used to study the convergence of an empirical distribution from the i.i.d. samples to the true
distribution. Specifically, Barrio et al. \cite{barrio1999_CLT-wass} showed that under the Wasserstein metric the empirical distribution converges to the true distribution almost surely as the number of samples go to infinity. Fournier and Guillin \cite{fournier2014_rate-convg-wass}
have further shown that if the true distribution $P_{\textrm{true}}$ has a light tail,  then the probability $\textrm{Pr}\big\{ W\big( P_0, P_{\textrm{true}} \big)\ge r_0 \big\}$ can be bounded from above by a function that decays exponentially with the sample size and $r_0$.

The \eqref{opt:DRO} modeling framework allows to protect against ambiguity in the distribution when arriving at a decision. 
It also allows to perform sensitivity analysis with respect to the empirical distribution. This is useful particularly when the sample size $m$ is small.

\subsection*{Contributions and organization of this paper} 
This paper makes the following contributions:
\begin{enumerate}
	\item It is common in robust optimization to dualize the inner problem to develop a reformulation of the original model \cite{book_robust_opt2009,Ben-Tal-1998_rob-conv-opt,Ben-Tal2002_rob-opt-meth-app,Ben-Tal2002_rob-sol-QCQP}. The definition of Wasserstein metric in \eqref{def:Kantorovich_metric} uses a semi-infinite number of equality constraints, therefore its direct use is not suitable for dualization of the inner problem. We  prove in Section~\ref{sec:reform} that the inner problem in (\ref{opt:DRO}) is equivalent to a conic linear optimization problem. We show that this conic program can be dualized with no duality gap, thus obtaining a semi-infinite programming reformulation of (\ref{opt:DRO}). The constraint sets in this semi-infinite program decompose in the observed samples.  The decomposition allows for an independent separation problem for each observed sample $\xi^i$ for algorithms that use cuts to solve the reformulated semi-infinite program. The results in Section~\ref{sec:reform} only assume that $h(\theta,\xi)$ is a continuous and bounded function in $\theta$, and $\Xi$ is a compact set (See Assumption~\ref{ass:theta_Xi_compact_h_bounded}). These results, for example, are applicable for problems involving mixed-integer variables used in the definition of $\Xi$. 
	
	\item We adapt the exchange algorithm and the central cutting-surface algorithm from \cite{hettich1993_SIP-thy-methd-appl} and \cite{mehrotra2015} in Sections~\ref{sec:exchange_method} and \ref{sec:cutting-surface-alg} for the general and convex cases, respectively. We show finite convergence of the exchange method to a solution with a desirable accuracy. The linear rate of convergence for the cutting surface algorithm presented here exploits the structure of the semi-infinite program.
Specifically, a global linear rate of convergence is proved. 	
	
	\item In Section~\ref{sec:comp-perform} we present results on the computational performance of the central cutting-surface algorithm for solving WRLR problems. We find that the number of oracle calls are typically $20\sim 50$, and the number of cuts added to the model are typically $3\sim 10$ times the number of training samples. The solution time is $\lesssim 100$ times that of solving the ordinary logistic regression model.  
	
	In Section~\ref{sec:forecast-perform}, we present performance results on the quality of model obtained by the Wasserstein-robust logistic regression (WRLR) approach and compare it with the performance of the ordinary logistic regression (LR). Our motivation is to study a setting in which the number of available samples is small, and robustness is used to understand and possibly improve the quality of the trained model. This is typically the case at an early
stage of a study (e.g., in healthcare) when limited data is available due to data collection expenses. We use $m$ to be $\{50, 75, 100, 150\}$ in the numerical experiments to test the performance of the Wasserstein-robust logistic regression model (See Section~\ref{sec:num_exp}). Eleven data sets from UCI Machine Learning Repository are used. We use area under the receiver operator characteristic curve (AUC) to evaluate the performance of the models \cite{Zweig1993_ROC-rev}. We find that the Wasserstein-robust logistic regression WRLR model has a significantly better out of sample performance than logistic-regression model (with $\alpha=0.05$) in 24 $(55\%)$ cases. The predictive performance of WRLR is worse in 7 $(16\%)$ cases ($\alpha=0.05$), and for the remaining 13 $(29\%)$ cases the difference is not statistically significant. The WRLR models also have smaller standard error when compared to logistic regression, suggesting that the model is more robust. 
 
\end{enumerate}

\section{Literature Review}
\label{sec:liter_rev}
We now provide a literature review of prior work on distributionally robust optimization (DRO) and semi-infinite programming.
These topics are relevant because the \eqref{opt:DRO} problem is reformulated as a semi-infinite program, and a separation oracle is 
needed in the algorithms. 

    \subsection{Distributionally robust optimization}
    Distributional robust optimization is a generalization of robust optimization (RO) \cite{bertsimas2013_data-driven-rob-opt,goh2010_DRO-tract-approx,xu2012-DR-interpret-rob-opt}, where an ambiguity set is used to model problem data distribution. The use of an ambiguity set in distributionally robust optimization overcomes the weakness of traditional robust optimization framework, as the RO model is often considered to be very conservative. 
Even when a deterministic model is a convex optimization problem, its DR-counterpart is NP-hard in most cases \cite{book_robust_opt2009}.     
Therefore, literature on RO and DRO either makes assumptions on the function form of 
the objective and constraints with respect to uncertain parameters to ensure the convexity of the model \cite{Ben-Tal-1999_rob-lin-prog,bertsimas2005_opt-ineq-prob}, 
or purposes convex reformulations to approximate original problems \cite{goh2010_DRO-tract-approx,wiesemann2015_dist-rob-conv-opt}.  
    
Studies of DRO focus on ways of defining the ambiguity set, reformulations of these models into computationally tractable problems, probability guarantee of the constraint satisfaction by the true distribution, and applications. We briefly review these aspects and the literature below.    
    \subsubsection{The ambiguity set}	
    A common approach to describe ambiguity sets is to use moments of the distribution followed by the random parameters \cite{bertsimas2010_models-minmax-stoch-LP-risk-aver,birge1987,delage2010_distr-rob-opt-mom-uncer, dupacova1987_minmax-stoch, mehrotra2015, book_stoch_prog_1995, shapiro2004_class-of-minimax-analysis-stoch-progs, shapiro2002_minimax-analysis-stoch-probs}. Bertsimas and Popescu \cite{bertsimas2005_opt-ineq-prob} discuss properties of probability distributions satisfying such constraints. It is also possible to use a \emph{statistical distance} as a way of measuring the difference of two probability distributions. The statistical distances used in the DRO models are Wasserstein metric \cite{esfahani2015_distr-rob-wass-metric,mehrotra2014_model-algthm-distr-rob-least-square,pflug2012_1-over-N-invest-optimal-under-ambiguity, pflug2007_ambig-in-portfolio-selection, s-abadeh2015_distr-rob-log-reg,wozabal2012_framework-opt-under-ambiguity}, $\phi$-divergence, $\chi^2$-distance, Kullback-Leibler divergence \cite{ben-tal2013_rob-opt-uncert-prob,calafiore2007_amb-risk-meas-opt-rob-portf,jiang2015_risk-aver-two-stage-stoch-prog-distr-amb,love2016_phi-diverg-constr-amb-stoch-prog-data-driven-opt,wang2016_likehd-rob-opt-data-driven,yanikoglu2013_safe-approx-amb-constr-hist-data}, and the Prokhorov metric \cite{erdogan2006_ambig-chance-constr-rob-opt}. 
     \subsubsection{Reformulation of DRO models}
      Shapiro and Kleywegt \cite{shapiro2002_minimax-analysis-stoch-probs} show that under mild regularity conditions, the DRO problem with a deterministic set of constraints is equivalent to a stochastic programming problem based on a probability distribution that is a linear combination of distributions in the ambiguity set of the original DRO problem. It means that solving a DRO problem in this case is equivalent to solving a stochastic programming problem.
Goh and Sim \cite{goh2010_DRO-tract-approx} investigate a two-stage DRO model whose objective has a linear structure, but can be used to express piece-wise linear utility functions and CVaR constraints. They show that by restricting the recourse variables to be affine mappings of uncertain parameters, this two-stage DRO model can be reformulated as a minimax linear programming problem. Delage and Ye \cite{delage2010_distr-rob-opt-mom-uncer} show that DRO problems whose ambiguity sets are defined by the first and second moment inequalities are polynomial-time solvable under the assumption that the objective is convex in the decision variables and concave in uncertain parameters. They also provide semidefinite formulations for the data-driven problems. 
As an often used objective, the least-square loss function is convex in both decision variables and model parameters, hence violating the assumption in \cite{delage2010_distr-rob-opt-mom-uncer}. To overcome this obstacle, Mehrotra and Zhang \cite{mehrotra2014_model-algthm-distr-rob-least-square} give conic and semidefinite programming reformulations of DR-least-squares problems with ambiguity sets defined using the first two moments,
Wasserstein metric, and bounds on the probability density functions, respectively. 
Mehrotra and Papp \cite{mehrotra2015} use the central cutting-surface algorithm they developed 
to solve DRO models where the ambiguity set is specified using arbitrarily many generalized moments.
Wiesemann et al. \cite{wiesemann2015_dist-rob-conv-opt} propose a framework for modeling and solving distributionally robust convex optimization problems, in which the ambiguity set is conically representable and constraint functions are piece-wise affine in both decision variables and random parameters. They show that the reformulated problem is polynomial-time solvable under a strict nesting condition of the confidence sets. Esfahan and Kuhn \cite{esfahani2015_distr-rob-wass-metric} show that using the conic duality theory, the data-driven DRO problem with a Wasserstein-metric ambiguity set can be reduced to finitely many tractable convex optimization problems, if the loss function can be expressed as the point-wise maximum of finitely many concave functions in the uncertain parameter. However, this assumption on the loss function is violated by many statistical learning models, such as the logistic regression model considered in the computational section of this paper. 

Many DRO problems involve robust chance constraints, which are often non-convex. Jiang and Guan \cite{jiang2016_data-driven-chance-constr-stoch-program} study DR-chance constraints defined by the \newline $\phi$-divergences. They show that these constraints are equivalent to ordinary chance constraints based on some nominal probability measure. Hanasusanto et al. \cite{hanasusanto2015_DRO-uncert-quant-chance-constr-prog} show that if the ambiguity set in the robust chance constraint is defined by moments and satisfies a nested condition, the worst-case probability is an optimal solution of a conic optimization problem. 
For a recent review on tractable reformulations of robust chance constraints, see \cite{postek2016_comp-tract-counterparts-DRO-risk-measures}. 

     \subsubsection{The probability bound} 
      Studies on the probability that the true distribution is contained in the ambiguity set are related to the ambiguous chance constrained programming \cite{calafiore2006_DRO-chance-constr-LP,hanasusanto2015_DRO-uncert-quant-chance-constr-prog,zymler2013-DR-joint-chan-constr-sec-ord-mom-inf}. Erdo{\u g}an and Iyengar \cite{erdogan2006_ambig-chance-constr-rob-opt} show that the sampling of robust constraints is a good approximation for the DR-chance-constraint problem with a high probability. Delage and Ye \cite{delage2010_distr-rob-opt-mom-uncer} use the size of the ellipsoid confidence region using the second moment to satisfy a given level of probability guarantee. Esfahan and Kuhn \cite{esfahani2015_distr-rob-wass-metric} give a result on the out-of-sample performance guarantee of the solution to the data-driven DRO problem with a Wasserstein-metric ambiguity set. Calafiore and Ghaoui \cite{calafiore2006_DRO-chance-constr-LP} show that chance constraints of linear inequalities with respect to a radial distribution (i.e., Gaussian distribution and uniform distribution on ellipsoidal support) can be converted explicitly into convex second-order cone constraints. Additionally, they show that distributionally robust chance constraints of linear inequalities under a few important distribution families (distributions with known mean and covariance, radially symmetric non-increasing distributions, etc.) can be guaranteed by some deterministic convex constraints.  
Bertsimas et al. \cite{bertsimas2013_data-driven-rob-opt} propose a novel scheme of constructing uncertainty sets for data-driven robust optimization problems using hypothesis tests. The resulting model is computationally tractable and has application insights regarding using statistical estimate for chance constraint violation. Ben-Tal et al. \cite{ben-tal2013_rob-opt-uncert-prob} show that the robust counterpart of linear optimization problems with uncertainty set defined by $\phi$-divergence are tractable for most choices of $\phi$. Constructing confidence sets using $\phi$-divergence is also studied in \cite{love2016_phi-diverg-constr-amb-stoch-prog-data-driven-opt,yanikoglu2013_safe-approx-amb-constr-hist-data}.

      \subsubsection{Applications of DRO Models}
      Modeling and solutions of the distributionally robust counterparts of deterministic optimization models are investigated in various areas in operations research, including but not limit to: inventory management \cite{zhang2016_DRO-two-stage-lot-sizing},  scheduling and logistics \cite{jaillet2016_routing-opt-under-uncert,liao2013_DR-workforc-scheduling-call-center-uncert-arriv-rate}, and risk management \cite{long2014_DR-discrete-opt-entrop-VaR,natarajan2008_incorp-asym-distr-inf-rob-VaR-opt}. In statistical learning, Lee and Mehrotra \cite{Lee2015_distr-rob-svm} study  distributionally robust linear support vector machine (DR-SVM) models with a Wasserstein-metric ambiguity set. They find that the (DR-SVM) model can be reformulated as a semi-infinite program, in which the master problems are convex quadratic programs and separation problems are linear programs. They also find that (DR-SVM) models have improved generalization capabilities than ordinary (SVM) models. Shaeezadeh-Abadeh et al. \cite{s-abadeh2015_distr-rob-log-reg} investigate a distributionally robust logistic regression model using a Wasserstein-metric ambiguity set. For the model in \cite{s-abadeh2015_distr-rob-log-reg}, the uncertainty set of the attribute vector is assumed to be the entire $\mathbb{R}^n$, and the authors show that  the semi-infinite constraints in the dual problem are equivalent to a single constraint obtained using the conjugate function. The assumption in \cite{s-abadeh2015_distr-rob-log-reg} is not practical in settings such as healthcare, where certain physiological variables (e.g., heart rate, blood pressure) must necessarily be bounded. In the framework of this paper, the uncertainty set is assumed to be compact. Additionally, feature variables can be integer-valued.

\subsection{Theory and numerical methods for semi-infinite programming}
The study of distributional robust optimization models considered in this paper benefit from the known literature on semi-infinite programming.
Semi-infinite programming (SIP) problems are optimization problems with constraints induced by a continuous parameter.
The study of SIP is initialized by the work of Haar~\cite{haar1924_uber-lin-ung} and Charnes~\cite{charnes1962_duality-haar-progs-finite-seq-space,charnes1963_duality-SIP,charnes1963_SIP-thy-KKT} focusing on linear-SIP problems. Later, the first and second order optimality conditions of general SIP were given in \cite{hettich1977_first-sec-order-cond-local-opt-finite-dim,hettich1978_SIP-opt-cond-appl,hettich1995_sec-ord-opt-cond-gen-SIP,nuernberger1985_global-unicity-opt-approx,nuernberger1985_global-unicity-SIP,still1999_gen-SIP-thy-methd}. For reviews of the theory and methods for SIP, see \cite{hettich1993_SIP-thy-methd-appl,lopez2007,reemtsen1998_num-methd-SIP-survey}.    
 
Numerical methods for solving convex SIP problems include primal methods \cite{wang2015_feasible-methd-SIP}, dual methods \cite{hettich1993_SIP-thy-methd-appl}, penalty methods \cite{lin2014_new-exact-penalty-methd-SIP,yang2015_opt-cond-SIP-lp-exact-penalty-func}, smooth approximation and projection methods \cite{xu2014_solving-SIP-smooth-proj-grad-methd}, and cutting-plane methods \cite{kortanek1993}. Primal methods are based on searching for feasible descent directions, while dual methods are based on finding a solution of the system of KKT optimality conditions. In a penalty method, constraints are penalized in the objective and the penalty term is an integral of the constraint function over the continuous parameter. In a smooth approximation and projection method, infinitely many functions are replaced by a integral entropy function as an approximation. 
The SIP is solved by the smoothing projected gradient method. In a cutting-plane algorithm, a typical iteration involves solving a master problem with finitely many constraints and adding violated constraints obtained from solving a separation problem. Mehrotra and Papp \cite{mehrotra2015} developed a central-cutting-surface algorithm with a linear rate of convergence for solving convex SIP problems. They demonstrate that adding cutting surfaces, as compared to cutting planes, can be computationally effective for problems in high dimension and ensure greater stability in the algorithm's performance. The central-cutting surface method
is related to the exchange method \cite{hettich1993_SIP-thy-methd-appl}, however it uses a centrality parameter in the algorithm.

\section{Reformulation of the Wasserstein-robust Optimization Problem}\label{sec:reform}
In this section we show that \eqref{opt:DRO} is equivalent to a semi-infinite program. This semi-infinite program decomposes in scenarios. We note that the definition of $\W(P, P_0)\le r_0$ involves infinitely many equality constraints of the form $K(A\times \Omega) = P,\ \forall A \in \mathcal{F}$. Since this form is not suitable for dualization of the inner problem, Theorem~\ref{thm:problem_equivalence} below reformulates the inner problem in (\ref{opt:DRO}) as a conic linear program with finitely many constraints. Proposition~\ref{lem:fP_continuity_over_P} provides an intermediate result, and shows the continuity of the objective function of the inner problem in (\ref{opt:DRO}) with respect to the probability measure. Lemma~\ref{lem:P_converge} is a technical result that generalizes similar convergent sequence existence results for the finite dimensional sets to a set defined by the Wasserstein metric. We make the following general assumption throughout this paper. Additional assumptions are made at appropriate places as results are developed. Note that results in this section make no assumption on the metric $d(\cdot,\cdot)$ used in defining \eqref{def:Kantorovich_metric}.  
\begin{assumption}
\label{ass:theta_Xi_compact_h_bounded}
We assume that $r_0>0,\;\forall \theta\in\Theta$ in \eqref{def:Kantorovich_metric}. The feasible region $\Theta\subseteq\mathbb{R}^n$ and the domain $\Xi$ are compact. The function $h(\cdot,\cdot)$ is bounded on $\Theta\times\Xi$, and for every $\theta\in\Theta$, there exists $C(\theta)>0$ such that the function $h(\theta,\cdot)$ satisfies $\lvert h(\theta,s_1)-h(\theta,s_2) \rvert\le C(\theta) d(s_1,s_2),\;\forall s_1,s_2\in\Xi$. 
\end{assumption} 

\begin{definition}
Let $(\mathcal{X},d_{\mathcal{X}})$ be a metric space formed by defining a metric $d_X$  on the topological space $\mathcal{X}$. A function $f:\;\mathcal{X}\longmapsto\mathbb{R}$ is continuous in $(\mathcal{X},d_{\mathcal{X}})$ if for a
given $x\in \mathcal{X}$, and $\varepsilon>0$, $\exists\delta>0$ such that $\forall x^{\prime}\in \mathcal{X}$, 
$d_{\mathcal{X}}(x^{\prime},x)<\delta$ we have $|f(x^{\prime})-f(x)|<\varepsilon$. 
\end{definition}

Lemma~\ref{lem:P_converge} below shows that the interior of the Wasserstein ball has a sequence of distributions that converge to a chosen point
on the boundary of the ball. We need the following two results, proved in Appendix~\ref{appendix:fcontP}, in the proof of \newline Lemma~\ref{lem:P_converge}.
\begin{proposition}
\label{lem:fP_continuity_over_P}
Let Assumption~\ref{ass:theta_Xi_compact_h_bounded} hold, then the function $f(\theta,\cdot):$ \newline
$\mathcal{M}(\Xi,\mathcal{F})\longmapsto\mathbb{R}$ defined by $f(\theta,P):=\mathbb{E}_{P}[h(\theta,\xi)]$ is 
 continuous in $(\mathcal{M}(\Xi,\mathcal{F}),\W)$.
\end{proposition}

\begin{proposition}
\label{prop:int_d_K_zero}
For any probability measure $P\in \mathcal{M}(\Omega,\mathcal{F})$, there exists a $K\in \mathcal{M}(\Xi\times\Xi,\mathcal{F}\times\mathcal{F})$ such that $K(A\times A)=P(A)$ and $K(\Xi\times A)=K(A\times\Xi)=P(A)$ for any $A\in \mathcal{F}$. For such a joint probability measure $K$, we have
\begin{equation}
\int_{(s_1\times s_2)\in\Xi\times\Xi}d\big(s_1,s_2\big) K(ds_1\times ds_2)=0. 
\end{equation} 
\end{proposition}

\begin{lemma}
\label{lem:P_converge}
Let Assumption~\ref{ass:theta_Xi_compact_h_bounded} hold, and $\mathcal{P}$ be defined as in \emph{(\ref{def:ambiguity_set})}. Let $\mathcal{P}^{\prime}:=\{P\in\mathcal{M}(\Xi,\mathcal{F}):\; \W(P, P_0)<r_0 \}$ be the interior of $\mathcal{P}$. Then for any $P\in\mathcal{P}$, there exists a sequence $\{ P^n\}^{\infty}_{n=1}\subseteq \mathcal{P}^{\prime}$ such that $\lim_{n\to\infty}\;\W(P^n,P)=0$.
\end{lemma}
\begin{proof}
For any given $P\in\mathcal{P}$ and for any $\varepsilon>0$, by the definition of the Wasserstein metric,  there exists $K^{\varepsilon} \in \mathcal{M}(\Xi\times\Xi, \mathcal{F}\times \mathcal{F})$ such that $K^{\varepsilon}(\Xi\times A)=P_0(A)$, $K^{\varepsilon}(A\times\Xi)=P(A)$,\;$\forall A\in\mathcal{F}$, and 
\begin{equation}\label{eqn:prop_Kvare_le_e_vare}
\int_{(s_1\times s_2)\in\Xi\times\Xi}d(s_1,s_2)\; K^{\varepsilon}(ds_1\times ds_2) \le \W\big( P_0, P \big) +\varepsilon.
\end{equation}  
Define $K_0\in \mathcal{M}(\Xi\times\Xi,\mathcal{F}\times \mathcal{F})$ such that $K_0(A\times B)=P_0(A\cap B),\;\forall A,B\in \mathcal{F}$. By Proposition~\ref{prop:int_d_K_zero}, we have 
\begin{equation}\label{eqn:prop_int _d_K0_=zero}
\int_{(s_1\times s_2)\in\Xi\times\Xi}d(s_1,s_2) K_0(ds_1\times ds_2)=0.
\end{equation} 	
Now define $\{P^n \}^{\infty}_{n=1}\subseteq \mathcal{M}(\Xi,\mathcal{F})$ as: $P^n:=\lambda_n P + (1-\lambda_n) P_0$ with $\lambda_n\in(0,1)$ and $\lambda_n\to 1$. Define $K^{\varepsilon}_n:=\lambda_n K^{\varepsilon} + (1-\lambda_n) K_0$ as a probability measure in $\mathcal{M}(\Xi\times\Xi, \mathcal{F}\times\mathcal{F})$. It is straightforward to verify that $K^{\varepsilon}_n(\Xi\times A)=P(A)$ and $K^{\varepsilon}_n(A\times\Xi)=P^n(A)$,\;$\forall A\in\mathcal{F}$, using their definitions, which means the joint measure $K^{\varepsilon}_n$ satisfies marginal conditions with respect to $P$ and $P^n$.
First, we need to verify that $\big\{ P^n \big\}^{\infty}_{n=1}\subseteq \mathcal{P}^{\prime}$. To see this, we have
\begin{align}\label{eqn:W<lambdan-W}
&\W\big(P^n,P_0\big)\le \int_{(s_1\times s_2)\in\Xi\times\Xi} d(s_1,s_2)\;K^{\varepsilon}_n(ds_1\times ds_2) \nonumber\\
&= \lambda_n \int_{\Xi\times\Xi} d(s_1,s_2) \;K^{\varepsilon}(ds_1\times ds_2) + (1-\lambda_n) \int_{\Xi\times\Xi} d(s_1,s_2)\; K_0(ds_1\times ds_2) \nonumber\\
&\le \lambda_n\big[ \W\big(P, P_0 \big) + \varepsilon\big]. \hspace{1.5cm}\textrm{using (\ref{eqn:prop_Kvare_le_e_vare}-\ref{eqn:prop_int _d_K0_=zero})} 
\end{align}		
Since $\varepsilon$ can be chosen arbitrarily, we set 
\begin{equation}\label{eqn:epsilon-lambda-W}
\varepsilon = \textrm{min}\Bigg\{1, \; \frac{1}{2}\left( \frac{1}{\lambda_n} -1\right) \W\big(P, P_0 \big) \Bigg\}.
\end{equation}
Substituting (\ref{eqn:epsilon-lambda-W}) into (\ref{eqn:W<lambdan-W}) yields $\W\big(P^n,P_0\big)\le (1/2 + \lambda_n/2)  \W\big(P, P_0 \big) < r_0$, hence $\big\{ P^n \big\}^{\infty}_{n=1} \subseteq \mathcal{P}^{\prime}$.

It remains to verify that $\lim_{n\to\infty} \W(P^n,P)=0$. To see this, define $K\in\mathcal{M}(\Xi\times\Xi, \mathcal{F}\times\mathcal{F})$ such that $K(A\times B)=P(A\cap B),\;\forall A,B\in\mathcal{F}$. By Proposition~\ref{prop:int_d_K_zero}, we have 
\begin{equation}\label{eqn:int_K=0}
\int_{(s_1\times s_2)\in\Xi\times\Xi}d(s_1,s_2) K(ds_1\times ds_2)=0.
\end{equation}
Let $\widetilde{K}^{\varepsilon}_n := \lambda_n K + (1-\lambda_n) K^{\varepsilon}$ be a joint probability measure in $\mathcal{M}(\Xi\times\Xi, \mathcal{F}\times\mathcal{F})$. Then we have $\widetilde{K}^{\varepsilon}(\Xi\times A)=P(A)$ and $\widetilde{K}^{\varepsilon}(A\times \Xi)=P^n(A),\;\forall A\in\mathcal{F}$, which means that $\widetilde{K}^{\varepsilon}$ satisfies marginal conditions with respect to $P$ and $P^n$. It follows that
\begin{equation}
\begin{aligned}
&\W\big(P^n,P\big) \le \int_{(s_1\times s_2)\in\Xi\times\Xi} d(s_1,s_2)\;\widetilde{K}^{\varepsilon}_n(ds_1\times ds_2) \\
&= \lambda_n \int_{\Xi\times\Xi} d(s_1,s_2)\;K(ds_1\times ds_2)  + (1-\lambda_n) \int_{\Xi\times\Xi} d(s_1,s_2) \; K^{\varepsilon}(ds_1\times d s_2) \\
&\le (1-\lambda_n) \big[  \W\big( P, P_0 \big) + \varepsilon \big]. \hspace{1.5cm} \textrm{using (\ref{eqn:prop_Kvare_le_e_vare}),(\ref{eqn:int_K=0})}
\end{aligned}
\end{equation}	
Since $\lambda_n\to 1$, we have $\W\big(P^n, P\big) \to 0$ as $n\to \infty$.
\end{proof}

\begin{theorem}
\label{thm:problem_equivalence}
Let Assumption~\ref{ass:theta_Xi_compact_h_bounded} hold, and $\mathcal{P}$ be defined as in \emph{(\ref{def:ambiguity_set})}. For a given $\theta$,  the inner problem $\underset{P\in\mathcal{P}}{\emph{sup}}\; \mathbb{E}_P[h(\theta,\xi)]$ 
has a finite optimal value, and it is equivalent to the following  conic linear program \emph{(CLP):}
\begin{equation}
{
\begin{array}{cll}
 \underset{\mu}{\emph{sup}} &  \displaystyle \int_{s\in\Xi} h(\theta,s) \mu(ds\times\Xi) &  \\ 
\emph{s.t.} & \displaystyle \mu(\Xi\times\Xi^i)  = 1/m, &  i\in [m]  \\
& \displaystyle \mu(\Xi\times\Xi^{m+1}) = 0, \\
& \displaystyle \sum_{i\in[m]} \int_{s\in\Xi}  d(s,s^i) \mu(ds\times\Xi^i)\le r_0, & \\
& \displaystyle \mu \succeq 0, &   
\end{array}
}
\tag{CLP} \label{opt:JDMP}
\end{equation}   
where $\mu\succeq 0$ denotes that $\mu$ is a positive measure.
\end{theorem}
\begin{proof}
The inner problem in \eqref{opt:DRO} can be written as:
\begin{equation}\label{opt:inner_problem_original}
\begin{array}{crclcl}
\underset{P\in \mathcal{M}(\Xi,\mathcal{F})}{\textrm{sup}} & \multicolumn{3}{l}{ \displaystyle \int_{s\in\Xi} h(\theta,s) P(ds)} \\ 
\textrm{s.t.} &\displaystyle   \W(P,P_0) \le r_0. & & 
\end{array}
\end{equation}
Let val(eqn\#) denote the optimal value of a problem given by (eqn\#). By Assumption~\ref{ass:theta_Xi_compact_h_bounded} $\Xi$ is compact, and
$h(\theta,\cdot)$ is bounded in $\Xi$. Hence, the objective $\mathbb{E}_P[h(\theta,\xi)]$ is finite, and therefore val(\ref{opt:inner_problem_original}) is finite. We first show that val(\ref{opt:inner_problem_original})$=$val(\ref{opt:JDMP}). For this purpose we consider the following auxiliary problem:
\begin{equation}\label{opt:inner_problem_alternative}
\begin{array}{crclcl}
\underset{P\in \mathcal{M}(\Xi,\mathcal{F})}{\textrm{sup}} & \multicolumn{3}{l}{\displaystyle \int_{s\in\Xi} h(\theta,s)\; P(ds)} \\ 
\textrm{s.t.} &\displaystyle   \W(P, P_0)  <  r_0, && 
\end{array}
\end{equation}
whose feasible set is the interior of $\mathcal{P}$.  We will prove that $\textrm{val(\ref{opt:inner_problem_original})} \ge \textrm{val(\ref{opt:JDMP})} \ge \textrm{val(\ref{opt:inner_problem_alternative})}$, and $\textrm{val(\ref{opt:inner_problem_original})} = \textrm{val(\ref{opt:inner_problem_alternative})}$, and hence val(\ref{opt:inner_problem_original})$=$val(\ref{opt:JDMP}).

We now show that $\textrm{val(\ref{opt:JDMP})}\ge \textrm{val(\ref{opt:inner_problem_alternative})}$. Let $\widetilde{P}$ be a feasible solution of (\ref{opt:inner_problem_alternative}). Since we have $\W(\widetilde{P},P_0)<r_0$, by the definition of the Wasserstein metric in (\ref{def:Kantorovich_metric}), there exists a $\widetilde{K}\in\mathcal{S}(\widetilde{P},P_0) $ satisfying: 
\begin{equation}
\int_{\Xi\times\Xi} d(s_1,s_2)\; \widetilde{K}(ds_1\times ds_2) = \sum_{i\in[m]} \int_{s\in\Xi} d(s,s^i)\; \widetilde{K}(ds\times\Xi^i) \le r_0. \nonumber
\end{equation}
Therefore, $\widetilde{K}$ is a feasible solution of (\ref{opt:JDMP}) with the objective value $\int_{s\in\Xi} h(\theta,s) \widetilde{K}(ds\times\Xi)$. 
Now observe that $\int_{s\in\Xi} h(\theta,s)\;\widetilde{K}(ds\times\Xi)=\int_{s\in\Xi} h(\theta,s)\; \widetilde{P}(ds)$, 
because $\widetilde{K}\in\mathcal{S}(\widetilde{P},P_0)$.
Consequently, for any sequence $\big\{ \widetilde{P}_k \big\}^{\infty}_{1}$ such that 
$\int_{s\in\Xi}h(\theta,s)\widetilde{P}_k(ds)\to\textrm{val}(\ref{opt:inner_problem_alternative})$, there exist $\big\{ \widetilde{K}_k \big\}^{\infty}_{1}$ satisfying 
$\int_{s\in\Xi} h(\theta,s) \; \widetilde{K}_k(ds\times\Xi)=\int_{s\in\Xi}h(\theta,s)\widetilde{P}_k(ds)$, hence 
$\int_{s\in\Xi} h(\theta,s) \; \widetilde{K}_k(ds\times\Xi)\to\textrm{val}(\ref{opt:inner_problem_alternative})$.
It follows that $\textrm{val}(\textrm{\ref{opt:JDMP}})\ge \textrm{val}(\ref{opt:inner_problem_alternative})$. 
		
Next we show that val(\ref{opt:inner_problem_original}) $\ge$ val(\ref{opt:JDMP}). Suppose $\widehat{K}$ is a feasible solution of (\ref{opt:JDMP}). 
Let $\widehat{P}\in\mathcal{M}(\Xi,\mathcal{F})$ be the marginal distribution of $\widehat{K}$ such that $\widehat{P}(A):=\widehat{K}(A\times\Xi),\;\forall A\in\mathcal{F}$. Due to the constraints of (\ref{opt:JDMP}), $\widehat{P}$ satisfies $\W(\widehat{P},P_0)\le r_0$; and hence $\widehat{P}$ is a feasible solution of (\ref{opt:inner_problem_original}). Because $\widehat{K}\in\mathcal{S}(\widehat{P},P_0)$, we have $\int_{s\in\Xi} h(\theta,s)\;\widehat{P}(ds)=\int_{s\in\Xi} h(\theta,s)\;\widehat{K}(ds\times\Xi)$. Consequently, for any sequence $\big\{\widehat{K}_k\big\}^{\infty}_{1}$ such that \newline 
$\int_{s\in\Xi} h(\theta,s)  \widehat{K}(ds\times\Xi)\to\textrm{val(\ref{opt:JDMP})}$, there exist a sequence $\big\{\widehat{P}_k\big\}^{\infty}_{1}$ satisfying \newline
$\int_{s\in\Xi} h(\theta,s)\;\widehat{P}_k(ds)=\int_{s\in\Xi} h(\theta,s)\;\widehat{K}_k(ds\times\Xi)$. It follows that $\int_{s\in\Xi} h(\theta,s)\;\widehat{P}_k(ds)\to\textrm{val(\ref{opt:JDMP})}$, and hence $\textrm{val}(\ref{opt:inner_problem_original})\ge \textrm{val}(\textrm{\ref{opt:JDMP}})$. 
		
We now show that $\textrm{val(\ref{opt:inner_problem_original})}=\textrm{val(\ref{opt:inner_problem_alternative})}$. Since $\textrm{val}(\ref{opt:inner_problem_original})$ is finite, there exists a sequence of probability measures $\big\{ P_m \big\}^{\infty}_{m=1} \subseteq \mathcal{P}$  such that 
\begin{equation}\label{eqn:lim_int_l_P=val}
\lim_{m\to \infty} \textrm{ } \int_{s\in\Xi} h(\theta,s) \;P_m(ds) = \textrm{val(\ref{opt:inner_problem_original})},
\end{equation}		
where $\mathcal{P}:= \big\{  P\in \mathcal{M}(\Xi,\mathcal{F}): \W(P, P_0) \le r_0  \big\}$. Let
$\mathcal{P}^{\prime} := \big\{ P\in \mathcal{M}(\Xi,\mathcal{F}): \W(P, P_0) < r_0 \big\}$ be the feasible set of (\ref{opt:inner_problem_alternative}). By Lemma~\ref{lem:P_converge}, for any $P_m$ ($m\ge1$), there exists a sequence $\{P^n_m\}^{\infty}_{n=1} \subseteq \mathcal{P}^{\prime}$ such that $\lim_{n\to\infty}\W(P^n_m,P_m)=0$. Let $f(\theta,P):=\mathbb{E}_{P}[h(\theta,\xi)]$, then we have that
$\textrm{val(\ref{opt:inner_problem_original})}\ge \textrm{val(\ref{opt:inner_problem_alternative})}\ge f(\theta,P^n_m),\; \forall m,n\ge 1$.
Moreover, we also have that
\begin{equation}
\begin{aligned}
&\lim_{m\to\infty}\lim_{n\to\infty} |f(\theta,P^n_m)-\textrm{val}(\ref{opt:inner_problem_original})| \\
&\le \lim_{m\to\infty}\lim_{n\to\infty} \Big( \lvert f(\theta,P^n_m)- f(\theta,P_m) \rvert + |f(\theta, P_m)-\textrm{val}(\ref{opt:inner_problem_original})| \Big)  \\
&=0,  \hspace{1cm}\textrm{[using Proposition~\ref{lem:fP_continuity_over_P} and (\ref{eqn:lim_int_l_P=val}) ]}
\end{aligned}
\nonumber
\end{equation}
which implies that val(\ref{opt:inner_problem_alternative})$=$val(\ref{opt:inner_problem_original}).
\end{proof}
Note that the infimum used in (\ref{opt:JDMP}) is over the joint measure $\mu$. Moreover, in (\ref{opt:JDMP}) both the objective and constraints are in linear form of the measure $\mu$, which is a `new decision variable' in the inner problem. Theorem~\ref{thm:inner-problem-dualization} provides a dual of (\ref{opt:JDMP}) as a semi-infinite linear program in the linear space of signed measures using conic programming duality. This theorem shows that the dual problem has a formulation which decomposes in $\xi^i$, $i\in[m]$.

\begin{theorem}
\label{thm:inner-problem-dualization}
Let Assumption~\ref{ass:theta_Xi_compact_h_bounded} hold. The dual of \emph{(\ref{opt:JDMP})} can be written as the following  semi-infinite program:
\begin{equation}
\begin{array}{cll}
 \underset{v\in \mathbb{R}^{m+1}}{\emph{min}} & \displaystyle \frac{1}{m}\sum^m_{i=1} v_i + r_0\cdot v_{m+1} & \\ 
\emph{s.t.} &\displaystyle h(\theta,s) - v_i - v_{m+1}\cdot d(s,\xi^i)\le 0, & s\in\Xi, i\in[m]  \\    
& v_1,\ldots, v_m\in \mathbb{R}, \quad v_{m+1} \ge 0. &     
\end{array}
\tag{CLP-D}\label{opt:CSIP}
\end{equation}
Furthermore, strong duality holds, i.e., \emph{val(\ref{opt:JDMP})}=\emph{val(\ref{opt:CSIP})}. 
Additionally, \newline for $r_0>0$ the optimum solution of (\ref{opt:CSIP}) can be bounded by the following polytope:
\begin{equation}
\begin{aligned}
\mathcal{H}:=\Big\{v\in\mathbb{R}^{m+1}:\;&C_1 \le v_i\le (m+1)C_2 - m C_1,\textrm{ for }i\in[m], \\
 & 0\le v_{m+1}\le (C_2-C_1)/r_0 \Big\},
 \end{aligned}
\end{equation}
where $C_1$ and $C_2$ are lower and upper bounds of $h(\cdot,\cdot)$ on $\Theta\times\Xi$, respectively.
\end{theorem}
\begin{proof}
We note that (\ref{opt:JDMP}) can be rewritten as:
\begin{equation}
{\def\arraystretch{1.2}
\begin{array}{cll}
 \underset{\mu}{\textrm{sup}} & \displaystyle   \int_{(s_1\times s_2)\in\Xi\times\Xi} h(\theta,s_1)\; \mu(ds_1\times ds_2) &  \\ 
\textrm{s.t.}  &  \displaystyle \int_{(s_1\times s_2)\in\Xi\times\Xi} {\bf 1}_{\Xi^i}(s_2)\; \mu(s_1\times s_2)  = 1/m, &  i\in [m]  \\
&  \displaystyle \int_{(s_1\times s_2)\in\Xi\times\Xi} {\bf 1}_{\Xi^{m+1}}(s_2)\; \mu(s_1\times s_2)  = 0, \\
&  \displaystyle  \int_{(s_1\times s_2)\in\Xi\times\Xi} \Big[ \sum_{i\in[m]} d(s_1,s_2) \cdot {\bf 1}_{\Xi^i}(s_2)\Big] \; \mu(ds_1\times ds_2)\le r_0, & \\
&  \displaystyle  \mu \succeq 0, &   
\end{array}
}
\label{opt:JDMP_int_form}
\end{equation} 
where ${\bf 1}_{\Xi^i}(s)$ is an indicator function. Note that the second constraint in \eqref{opt:JDMP_int_form} is defined on the set $\Xi^{m+1}=\Xi\setminus (\cup_{i\in[m]} \Xi^i)$. For a given $\theta$, we define functions $\big\{\psi_j\big\}^{m+2}_{j=0}$ as follows:
\begin{equation}
\psi_j(s_1, s_2) = \left \{
			\begin{array}{ll}
			 h(\theta,s_1)  &  j=0 \\
			{\bf 1}_{\Xi^j}(s_2) & j\in [m+1] \\
			\sum^m_{i=1} d(s_1,s_2)\cdot {\bf1}_{\Xi^i}(s_2) & j=m+2.
			\end{array}
		\right.
\end{equation}
Clearly $\{\psi_j\}^{m+2}_{j=0}$ are bounded $\mathcal{F}\times\mathcal{F}$-measurable functions on $\Xi\times\Xi$. It follows that $\{\psi_j\}^{m+2}_{j=0}$ are $\mu$-integrable for any $\mu\succeq 0$. We will first put (\ref{opt:JDMP_int_form}) in a standard form of conic linear program. Let $\mathcal{X}$ be the linear space of \emph{finite signed measures}, and $\mathcal{X}^+:=\big\{ \mu\in \mathcal{X}: \textrm{ } \mu\succeq 0 \big\}$ be the set of non-negative measures which is a convex cone in $\mathcal{X}$. Let $\mathcal{X}^{\prime}$ be the set of functions that are $\mu$-integrable for all $\mu\in\mathcal{X}^+$, i.e.,
$
\mathcal{X}^{\prime}:=\{ f:\Xi\times\Xi\rightarrow\mathbb{R} \textrm{ }|\textrm{ } f\in L^1(\Xi\times\Xi,\mu),\textrm{ }\forall \mu\in\mathcal{X}^+ \}. 
$
Define $\langle \mu, f \rangle := \int_{(s_1\times s_2)\in\Xi\times\Xi} f(s_1,s_2) \mu(ds_1\times ds_2),\; \forall \mu\in\mathcal{X}, f\in\mathcal{X}^{\prime}$.
Define the linear operator $\mathcal{A}:\mathcal{X}\rightarrow \mathbb{R}^{m+2}$ as
\begin{equation}\label{def:operator_A}
\mathcal{A}\mu:=\big[\langle \mu,\psi_1 \rangle,\ldots,\langle \mu,\psi_{m+2} \rangle \big]^T. 
\end{equation}
Define a vector $b$ as $b=\big[\underbrace{1/m, \ldots, 1/m}_m,0,r_0 \big]^T \in \mathbb{R}^{m+2}$, and a convex cone $\mathcal{K}:={\bf 0}^{m+1}\times (-\infty, 0]$  in $\mathbb{R}^{m+2}$. Using the above notations, (\ref{opt:JDMP_int_form}) can be rewritten as:
\begin{equation}
\begin{array}{cl}
\underset{\mu}{\text{sup}}  &  \displaystyle \langle \mu, \psi_0 \rangle  \\ 
\textrm{s.t.} &\displaystyle \mathcal{A}\mu - b  \in \mathcal{K}   \\
& \quad \mu \in  \mathcal{X}^+.   
\end{array}
\tag{CLP1}\label{opt:JDMP-conic-LP}
\end{equation}
Using Lemma~\ref{thm:duality_of_conic_linear_optimization}(a) in Appendix~\ref{appendix:dual-thm-conic-prog}  from \cite{shapiro2001_conic-lp}, we have the following dual of (\ref{opt:JDMP-conic-LP}):
\begin{equation}
\begin{array}{cl}
\underset{w}{\textrm{inf}}  & \displaystyle -b^T\cdot w  \\ 
\textrm{s.t.} & \mathcal{A}^*w+ \psi_0 \in -(\mathcal{X}^+)^*   \\
& \quad w \in \mathcal{K}^*,    
\end{array}
\label{opt:JDMP_dual_1}
\end{equation}
where $\mathcal{K}^*=\mathbb{R}^{m+1}\times(-\infty,0]$ is the dual cone of $\mathcal{K}$, 
and $-(\mathcal{X}^+)^*=\{f\in \mathcal{X}^{\prime}:\; \langle \mu, f \rangle\le 0, \forall \mu\succeq 0  \}$ is the polar cone of $\mathcal{X}^{+}$. The adjoint operator $\mathcal{A}^*$ acts on $\mathbb{R}^{m+2}$ as: $\mathcal{A}^*v=\sum^{m+2}_{i=1}v_i\psi_j$, $\forall v\in\mathbb{R}^{m+2}$. Therefore, the dual problem (\ref{opt:JDMP_dual_1}) can be expressed as follows:
\begin{equation}
\begin{array}{cll}
\underset{w}{\textrm{min}}  & \displaystyle -\frac{1}{m}\sum^m_{i=1} w_i - r_0\cdot w_{m+2} &  \\ 
\textrm{s.t.} & \displaystyle  \langle \mu,\psi_0 \rangle + \sum^{m+2}_{i=1} w_i\langle \mu, \psi_i \rangle \le 0, &  \forall \mu\in\mathcal{X}^+,  \\
& w_{m+2} \le 0. &     
\end{array}
\label{opt:JDMP_dual_2}
\end{equation}
We have
\begin{align}
\sum^{m+2}_{i=1} w_i\langle \mu, \psi_i \rangle + \langle \mu,\psi_0 \rangle
 &= \sum^m_{i=1} \int_{\Xi\times\Xi^i} \big[w_i+w_{m+2}\cdot d(s_1,\xi^i)+h(\theta,s_1)\big] \mu(ds_1\times ds_2) \nonumber\\
&\qquad + \int_{\Xi\times\Xi^{m+1}} \big[w_{m+1}+h(\theta,s_1)\big]\mu(ds_1\times ds_2). \label{eqn:sum_wi<mu,psi>} 
\end{align}
Since $\mathcal{F}$ contains all singletons in $\Xi$, all integrands on the RHS of (\ref{eqn:sum_wi<mu,psi>}) must be non-positive to ensure the constraint $\sum^{m+2}_{i=1} w_i\langle \mu, \psi_i \rangle + \langle \mu,\psi_0 \rangle \le 0$,  $\forall \mu\in\mathcal{X}^+$. Otherwise, we concentrate the measure $\mu$ on the points $(\hat{s}\times \xi^i)$, $i\in[m]$, or the set $(\hat{s},\Xi^{m+1})$ at which one of the integrands is positive to achieve a contradiction.  In other words, $\sum^{m+2}_{i=1} w_i\langle \mu, \psi_i \rangle + \langle \mu,\psi_0 \rangle \le 0$,  $\forall \mu\in\mathcal{X}^+$ are equivalent to the following constraints:
{%\setlength{\jot}{10pt} 
\begin{align}
w_i+w_{m+2}\cdot d(s,\xi^i)+h(\theta,s)&\le 0, \quad \forall s\in\Xi \quad i\in[m],
\label{eqn:dual_JDMP_constraint_1} \\
w_{m+1}+h(\theta,s) &\le 0, \quad  \forall s\in\Xi.\label{eqn:dual_JDMP_constraint_2}
\end{align}
}
Since $w_{m+1}$ does not appear in the objective function of (\ref{opt:JDMP_dual_2}), and constraint (\ref{eqn:dual_JDMP_constraint_2}) 
can be satisfied  by choosing $w_{m+1}$ sufficiently small, we can remove $w_{m+1}$ and constraint (\ref{eqn:dual_JDMP_constraint_2}) from the dual problem. By rewriting variables $w_i$ as $-v_i$ for $i\in[m]$ and $w_{m+2}$ as $-v_{m+1}$,  the dual problem (\ref{opt:JDMP_dual_2}) can be rewritten as (\ref{opt:CSIP}).

We now show that strong duality holds. We only need to verify that val(\ref{opt:JDMP}) is finite and show that Sol(\ref{opt:CSIP}) is nonempty and bounded \cite{shapiro2001_conic-lp} (see Lemma~\ref{thm:duality_of_conic_linear_optimization}(b), in Appendix~\ref{appendix:dual-thm-conic-prog}). 
First, since the objective in (\ref{opt:JDMP}) is an expectation of a bounded function, val(\ref{opt:JDMP}) is finite. Second, we note that the feasible set of (\ref{opt:CSIP}), denoted as Fsb(\ref{opt:CSIP}), is a closed convex subset in $\mathbb{R}^{m+1}$. 
Since $C_1\le h \le C_2$ on $\Theta\times\Xi$, $v=\big[\underbrace{C_2,\ldots,C_2}_m,0\big]^T$ is a feasible solution of (\ref{opt:CSIP}), we have $\textrm{val(\ref{opt:CSIP})}\le C_2$. Moreover, substituting $s=\xi^i$ in the $i^{\textrm{th}}$ constraint of (\ref{opt:CSIP}), we have  $v_i\ge h(\theta,\xi^i)\ge C_1,\;i\in[m] $, which means that $C_1$ is a lower bound on $v_i$, $i\in[m]$. To show that $\textrm{Sol(\ref{opt:CSIP})}$ is bounded, 
consider any $v^*\in\textrm{Sol(\ref{opt:CSIP})}$. Since $\textrm{val(\ref{opt:CSIP})}\le C_2$, 
we have $\frac{1}{m}\sum^m_{i=1} v^*_i + r_0 \cdot v^*_{m+1}\le C_2$.
It follows that 
\begin{equation}
\begin{aligned}
&v^*_{m+1} \le \frac{1}{r_0} \Big( C_2 - \frac{1}{m}\sum^m_{i=1} v^*_i  \Big) \le \frac{1}{r_0} (C_2 - C_1), \\
&v^*_i \le m(C_2+r_0\cdot v^*_{m+1})-\sum^{i-1}_{j=1}v^*_j-\sum^m_{j=i+1}v^*_j 
\le (m+1)C_2 - m C_1, \quad  i\in[m], 
\end{aligned}
\end{equation}
which provide upper bounds on $v^*_i$, $i\in[m+1]$. Therefore, Sol(\ref{opt:CSIP}) is bounded by the following compact set:
\begin{equation}\label{eqn:s_compact_support}
\begin{aligned}
\mathcal{H}:=\Big\{v\in\mathbb{R}^{m+1}:\;&C_1 \le v_i\le (m+1)C_2 - m C_1,\textrm{ for }i\in[m], \\
 & 0\le v_{m+1}\le (C_2-C_1)/r_0 \Big\}. 
\end{aligned}
\end{equation}
Now Fsb(\ref{opt:CSIP})$\cap \mathcal{H}$ is a non-empty compact set in $\mathbb{R}^{m+1}$, and $\textrm{Sol(\ref{opt:CSIP})}$ is a subset of
$\textrm{Fsb(\ref{opt:CSIP})}\cap\mathcal{H}$. By Weierstrass Theorem \cite{bertsekas2003_conv-anal-opt}, there exists an optimal solution
to (\ref{opt:CSIP}). Therefore, Sol(\ref{opt:CSIP}) is nonempty, bounded, and attains its optimum at a solution in $\mathcal{H}$.
\end{proof}
\begin{corollary}
\label{cor:SIP_form}
 Let Assumption~\ref{ass:theta_Xi_compact_h_bounded}  hold. The Wasserstein-robust optimization problem \eqref{opt:DRO} is equivalent to the following semi-infinite  program:
\begin{equation}
\begin{array}{cll}
 \underset{\theta,v}{\emph{min}} & \displaystyle \frac{1}{m}\sum^m_{i=1} v_i + r_0\cdot v_{m+1} & \\ 
\emph{s.t.} &\displaystyle  h(\theta,s) - v_i - v_{m+1}\cdot d(s,\xi^i) \le 0, & s\in\Xi, i\in[m]  \\    
& \theta\in\Theta, v\in\mathcal{H}.    
\end{array}
\tag{WRO-D}\label{opt:DRO-D}
\end{equation}
\end{corollary}  
\begin{proof}
From Theorem~\ref{thm:inner-problem-dualization}, the inner problem of \eqref{opt:DRO} is reformulated as a minimization problem with semi-infinite 
constraints. We can now combine \eqref{opt:CSIP} with the outer problem of \eqref{opt:DRO}, and have an equivalent combined formulation as \eqref{opt:DRO-D}. 
Since from Theorem~\ref{thm:inner-problem-dualization} the optimal solution of the inner problem is bounded in the polytope $\mathcal{H}$,
these additional constraints are added to \eqref{opt:DRO-D}.  
\end{proof}
We note that the constraints in \eqref{opt:DRO-D} decompose in the scenarios $\xi^i$, $i\in [m]$;
and for a given $s$, $d(s,\xi^i)$ is a constant.

\section{Algorithms for the (WRO) Refomulation}
\label{sec:alg-SDP}
Corollary~\ref{cor:SIP_form} shows that the Wasserstein-robust optimization problem \eqref{opt:DRO} is equivalent to \eqref{opt:DRO-D},
 which is a semi-infinite program. Any algorithm for solving a general semi-infinite program can now be applied to solve \eqref{opt:DRO-D}. 
 For a general continuous function $h(\theta,s)$ in $\theta$ we present a modification of the exchange algorithm \cite{hettich1993_SIP-thy-methd-appl} in Section~\ref{sec:exchange_method}, which ensures $\varepsilon$-optimality after solving a finite number of finitely constrained master problems. In the special case where $h(\theta,s)$ is a convex function of $\theta$, \eqref{opt:DRO-D} is a convex semi-infinite program. We adapt the cutting surface algorithm in \cite{mehrotra2015} for \eqref{opt:DRO-D} and use its structure to achieve a global linear rate of convergence.
 
Let $x=[\theta,v]$ be the decision variables, and define the following functions:
\begin{equation}
\begin{aligned}
f(x) &:= \frac{1}{m} \sum^m_{i=1} v_i + r_0\cdot v_{m+1},   \\
g_i(x,s) &:= h(\theta,s) - v_i - v_{m+1}\cdot d(s,\xi^i), \quad i\in[m]. 
\end{aligned}
\end{equation} 
The problem (\ref{opt:DRO-D}) can be rewritten as:
\begin{equation}
\begin{array}{cll}
 \underset{x}{\textrm{min}} & \displaystyle f(x) & \\ 
\textrm{s.t.} &\displaystyle g_i(x,s)\le0, & s\in\Xi, i\in[m]  \\    
& x\in X,  \\  
\end{array}
\tag{SIP}\label{opt:SIP}
\end{equation} 
where $X=\Theta\times \mathcal{H}$. Problem \eqref{opt:SIP} is a semi-infinite program. An approach to obtain a solution of such problems is to solve relaxation problems (master problems) with a finite number of constraints, and add a violated constraint obtained from solving a separation problem (defined for $\forall\;i\in[m]$) to tighten the formulation iteratively. In particular, the separation problem of \eqref{opt:SIP} for identifying a violated constraint at the solution 
$(\widetilde{x},\widetilde{v})$ of the current master problem can be written as follows:
\begin{equation}
\underset{s\in\Xi}{\textrm{max}} \; g_i(\widetilde{x}, s)=h(\widetilde{\theta},s) - \widetilde{v}_i - \widetilde{v}_{m+1}\cdot d(s,\xi^i), \quad \textrm{for }i\in[m]. 
\tag{Sep-$i$} 
\label{opt:infeasible_detection}
\end{equation}
The inequality generated from solving \eqref{opt:infeasible_detection} is called a feasibility cut. 

For clarity of notations, we consider the following general form of a semi-infinite program:
\begin{equation}
\begin{array}{cll}
 \underset{x}{\textrm{min}} & \displaystyle f(x) & \\ 
\textrm{s.t.} &\displaystyle g(x,t)\le 0, & t\in T, \\    
& x\in X,    
\end{array}
\tag{gen-SIP}\label{opt:gen-SIP}
\end{equation}
where $X\subseteq\mathbb{R}^{k_1}$ and $T\subseteq\mathbb{R}^{k_2}\times\mathbb{Z}^{k_3}$, allowing that $T$ may be defined as a mixed-integer set.

\subsection{A modified exchange algorithm for (WRO) model}
\label{sec:exchange_method}
We now assume that we have an oracle to solve the master and separation problems of \eqref{opt:gen-SIP}. 
The modified exchange algorithm given in Algorithm~\ref{alg:gen-SIP} allows an $\varepsilon$-optimal solution
 to \eqref{opt:gen-SIP}, when compared with the original exchange algorithm.
Theorem~\ref{thm:finite-converg} shows that our modified exchange method finds a solution of a desired accuracy in finitely many iterations.
 
\begin{definition}
\label{def:accuracy-gen-SIP}
A point $z_0\in Z$ is an $\varepsilon$-feasible solution of \eqref{opt:gen-SIP} if \newline $\underset{t\in T}{\emph{max}}\; g(z_0,t)\le \varepsilon$.
A point $z_0\in Z$ is an $\varepsilon$-optimal solution of \eqref{opt:gen-SIP} if $z_0$ is $\varepsilon$-feasible and $f(z_0)\le \emph{val}\eqref{opt:gen-SIP}$.
\end{definition}
\begin{theorem}
\label{thm:finite-converg}
Let $Z\times T$ be compact, and $g(z,t)$ be continuous on $Z\times T$. Suppose we have an oracle that generates an optimal solution of the problem $\underset{z\in Z}{\emph{min}}\;\{ f(z):\;  g(z,t)\le 0, \; t\in T^{\prime}   \}$ for any finite set $T^{\prime}\subseteq T$, and an oracle that generates an $\varepsilon$-optimal solution of the problem $\underset{t\in T}{\emph{max}}\; g(z, t)$ for any $z\in Z$ and $\varepsilon>0$. Then Algorithm~\ref{alg:gen-SIP} returns an $\varepsilon$-optimal solution of \eqref{opt:gen-SIP} in finitely many iterations. 
\end{theorem}
\begin{proof}
Since $g(z,t)$ is continuous on $Z\times T$, and $Z\times T$ is compact, it follows that $g(z,t)$ is uniformly continuous on $Z\times T$.
Therefore, there exists an $\alpha>0$ such that 
\begin{equation}\label{eqn:g-g}
|g(z^{\prime},t^{\prime}) - g(z,t)|\le \frac{\varepsilon}{2}, \;\;\;\; \textrm{if } \|z^{\prime}-z \| + \|t^{\prime} - t \| \le \alpha. 
\end{equation}   
First, we prove by contradiction that the algorithm terminates in finitely many iterations.
Suppose the algorithm generates infinite sequences $\mathcal{Z}=\{ z_k\}^{\infty}_{0}$ and $\mathcal{T}=\{ t_k\}^{\infty}_{0}$ without terminating.
We claim that $t_{k+1}\notin \cup^{k}_{i=1} B(t_i,\alpha)$ for every $k$, where $B(t_i,\alpha)$ is the closed ball of center $t_i$ and radius $\alpha$
in $\mathbb{R}^{k_2+k_3}$. If not, there exists $t_k$ such that $t_k\in B(t_i,\alpha)$ for some $i<k$. Using \eqref{eqn:g-g} and $t_i\in T_{z-1}$, we have $g(z_{k-1},t_k)  \le g(z_{k-1},t_i) + \frac{\varepsilon}{2}  \le \frac{\varepsilon}{2}$, indicating that the termination criteria is satisfied. 
Therefore, the above claim holds. Now consider the compact set: $\mathcal{B}=\cup_{\{t\in T\}} B(t,\alpha)$. 
Since $B(t_k,\alpha)\subseteq\mathcal{B}$ for every $k$, and by the claim we have $B(t_i,\alpha)\cap B(t_j,\alpha)=\emptyset$ 
for every $t_i,t_j\in\mathcal{T}$ with $i\neq j$, it follows that $\sum^{\infty}_{i=0}\textrm{vol}(B(t_i,\alpha))\le\textrm{vol}(\mathcal{B})$, 
which leads to a contradiction.

We now prove that once the algorithm terminates, it returns an $\varepsilon$-solution. 
Suppose it terminates at the end of the $n$th iteration. Then $z_n$ is an optimal solution
of the problem $\underset{z\in Z}{\textrm{min}}\;\{ f(z):\; g(z,t)\le 0, \; t\in T_n   \}$.
It follows that $f(z_n)\le \textrm{val}(T_n)  \le \textrm{val}\eqref{opt:gen-SIP}$,
where $\textrm{val}(T_n)$ is the optimal value of the problem \eqref{opt:gen-SIP} with constraint index set $T_n$. 
By the separation oracle and termination criteria, it also holds that $\underset{t\in T}{\textrm{max}}\; g(z_k, t)\le g(z_n,t_{n+1})+\frac{\varepsilon}{2}\le \varepsilon$. Hence, $z_n$ is an $\varepsilon$-optimal solution of \eqref{opt:gen-SIP}.
\end{proof}

\begin{algorithm}%[H] 
	\caption{A modified exchange algorithm to solve \eqref{opt:gen-SIP}. }
	\label{alg:gen-SIP}
	\begin{algorithmic}
	\State {\bf Prerequisites}: Two oracles specified in Theorem~\ref{thm:finite-converg}.
	\State {\bf Output}: An $\varepsilon$-optimal solution of (\ref{opt:gen-SIP}).
	 \State{\bf Step 1} Set $T_0\gets\emptyset$, $k\gets 0$.   
	 \State{\bf Step 2}  Determine an optimal solution $z_k$ of the problem 
	 				  $\underset{z\in Z}{\textrm{min}}\;\{ f(z):\; \textrm{s.t. } g(z,t)\le 0, \; t\in T_k   \}$.
   
	 \State{\bf Step 3}  Determine a $\frac{\varepsilon}{2}$-optimal solution $t_{k+1}$ of the problem $\underset{t\in T}{\textrm{max}}\; g(z_k, t)$.
	 				  If $g(z_k,t_{k+1})\le \frac{\varepsilon}{2}$, stop and return $z_k$; otherwise let $T_{k+1} \gets T_{k}\cup\{ t_{k+1}\}$, 
					  $k\gets k+1$ and go to Step 2	 
     
       	\end{algorithmic}	 
\end{algorithm}

\subsection{A central cutting-surface algorithm for the convex case}
\label{sec:cutting-surface-alg}
In this section we make the following additional assumption.
 \begin{assumption}
\label{ass:theta_convex_h_convex}
The feasible region $\Theta$ is convex, and the function $h(\cdot,s)$ is convex for every $s\in\Xi$. 
Furthermore, there exists a parameter $B$ satisfying the following condition:
\begin{equation}
\label{eqn:B_condition}
B>\|\eta\|, \quad \forall \; \eta\in\partial_{\theta} h(\theta,s) \;\; \forall \; \theta\in \Theta, \; s\in\Xi,  
\end{equation}
where $\partial_{\theta} h(\theta,s)$ is the sub differential set of $h(\theta,s)$ at $\theta$.
\end{assumption}  
Since the master problem of \eqref{opt:SIP} is a convex optimization problem, we assume it can be solved efficiently up to optimality. 
We now present a central cutting-surface algorithm to solve \eqref{opt:SIP}.
A pseudo-code for this algorithm is given in Algorithm \ref{alg:cutting-surface}. 
The algorithm is initialized with a solution $x^{(0)}=[\theta^0,{\boldsymbol 0}_{m+1}]$, where $\theta^0$ may be taken as a solution of the empirical deterministic optimization problem: 
\begin{equation}
\label{opt:EDO}
\underset{\theta\in\Theta}{\textrm{min}}\; \frac{1}{m}\sum^m_{i=1} h(\theta,\xi^i).
\tag{EDO}
\end{equation}
At the $k$th iteration of this algorithm (Step 1) a master problem with a centering argumentation is solved. 
This problem is defined by a set of constraints (with the index set $Q^{(k-1)}\subseteq\Xi$) inherited from the $(k-1)$th iteration.
The master problem at the $k$th iteration is formulated as follows:
		 \begin{equation}
		 \begin{array}{cll}
 		 \underset{x,w,t}{\textrm{max}} & w \\ 
		 \textrm{s.t.} & t + w \le M^{(k-1)}, &        \\
		 			& f(x) - t + B\cdot w \le 0,  &   \\
		                     &  g_i(x,s) + B\cdot w \le 0, & \forall i\in[m],\;\; s\in Q^{(k-1)}, \\
&   x\in X. &   
		\end{array}
		\tag{Master}\label{opt:master}
		\end{equation}
The (\ref{opt:master}) problem is a convex optimization problem. For clarity, we drop the index $i$ in \eqref{opt:master} in Algorithm~\ref{alg:cutting-surface}.
We assume that there is an oracle to find an $\varepsilon$-optimal solution to \eqref{opt:infeasible_detection}
for any $\varepsilon>0$. The algorithm terminates if no feasibility cut is found. Otherwise, the newly found feasibility cut is added to the working set (Step 4). At the end of each iteration, we may optionally drop certain constraints that are not binding at the current solution of the master problem (Step 6). 
\begin{algorithm}[H] 
	\caption{A central cutting-surface algorithm from \cite{mehrotra2015} to solve (\ref{opt:SIP}). }
	\label{alg:cutting-surface}
	\begin{algorithmic}
	\State {\bf Prerequisites}: Assumptions~\ref{ass:theta_Xi_compact_h_bounded} and \ref{ass:theta_convex_h_convex} hold. There exists an oracle to find an $\varepsilon$-optimal 
	solution to (\ref{opt:infeasible_detection}).
	\State {\bf Input}: A strict upper bound $U$ of the objective function $h$, a centering parameter $B>0$ which satisfies (\ref{eqn:B_condition}), a tolerance error $\varepsilon$, an arbitrary $\alpha>1$ specifying how aggressively cuts are dropped. 
	\State {\bf Output}: An $\varepsilon$-optimal solution to (\ref{opt:SIP}).
	
	 \State{\bf Step 1} (Initialization). Set $k\gets1$, $M^{(0)}\gets U$, $v^{(0)}\gets {\bf{0}}_{m+1}$, $x^{(0)}\gets [\theta^0, {\bf 0_{m+1}}]$, 
	 	\State\hspace{\algorithmicindent} $\widetilde{x}^{(0)}\gets x^{(0)}$, where $\theta^0$  is a solution to (\ref{opt:EDO}). Let $Q^{(0)}\gets\big\{ s^{(0)}\}$, where $s^{(0)}=\xi$.

	 \State{\bf Step 2} (Solve a master problem). Determine the optimal solution $(x^{(k)},w^{(k)})$ to \eqref{opt:master}.

	 \State{\bf Step 3} (Optimal soluton?). If $w^{(k)}=0$, stop and return $\widetilde{x}^{(k-1)}$.

        \State{\bf Step 4} (Feasible solution?). Find an $\varepsilon$-optimal solution denoted by $s^{(k)}$ to (\ref{opt:infeasible_detection}) 
        		\State\hspace{\algorithmicindent} using the oracle. If $g(x^{(k)},s^{(k)})>0$ and go to Step 5, otherwise go to Step 6.

       \State{\bf Step 5} (Add a cut). Set $Q^{(k)}\gets Q^{(k-1)}\cup \{s^{(k)}\}$, $\widetilde{x}^{(k)}\gets \widetilde{x}^{(k-1)}$ and $M^{(k)}\gets M^{(k-1)}$.
       			\State\hspace{\algorithmicindent} Go to Step 6.

       \State{\bf Step 6} (Update best know $\varepsilon$-feasible solution). Set $Q^{(k)}\gets Q^{(k-1)}$, $\widetilde{x}^{(k)}\gets x^{(k)}$ $M^{(k)}\gets f(x^{(k)})$.

       \State{\bf Step 7} (Optionally drop cuts). Let $D=\big\{ s^{(l)} \in Q^{(k)}:\; l\in\{0\}\cup[k] \;\big|\;$
        		\State\hspace{\algorithmicindent} $w^{(l)} \ge \alpha w^{(k)} \textrm{ and } g(x^{(k)}, s^{(l)} ) + B\cdot w^{(k)} < 0 \big\}$ and set $Q^{(k)}\gets Q^{(k)}\setminus D$.      
     
      \State\hspace{\algorithmicindent} Set $k\gets k+1$ and go to Step 2.
      
	\end{algorithmic}	 
\end{algorithm}
The convergence of Algorithm~\ref{alg:cutting-surface} is given by Theorem~\ref{thm:convergence-cutting-surf-alg}. 
This theorem is a refinement of the linear rate of convergence result proved in Theorem~8 of  \cite{mehrotra2015} (See Appendix~\ref{appendix1}) for the central cutting surface algorithm
when specialized to \eqref{opt:DRO-D}.
\begin{theorem}
\label{thm:convergence-cutting-surf-alg}
Convergence of the central cutting-surface algorithm.
\begin{enumerate}
	\item[\emph{(a)}] Algorithm~\ref{alg:cutting-surface} either finds an $\varepsilon$-optimal solution of \emph{(\ref{opt:SIP})} in finitely many iterations or generates $\{ \widetilde{x}^{(k)} \}^{\infty}_{k=1}$ that each accumulated point is an $\varepsilon$-optimal solution of \emph{(\ref{opt:SIP})}. 
	\item[\emph{(b)}] Algorithm~\ref{alg:cutting-surface} converges linearly in objective function value between consecutive feasibility cuts. The rate of convergence satisfies:
	\begin{equation}
		\rho \le 1- \frac{1}{2B^{\prime}+1},
	\end{equation}  
	where $B^{\prime}:=\emph{max}\{ \sqrt{r^2_0+1+(1/m)},\; \sqrt{B^2+L^2+1} \}$ and $L=\emph{max}\{d(s,\xi^i):\; s\in \Xi, \; i\in[m] \}$. 
\end{enumerate}
\end{theorem}
\begin{proof}
Note that \eqref{opt:SIP} is equivalent to the following reformulation:
\begin{equation}\label{opt:SIP-reform}
\begin{array}{cll}
 \underset{x,t}{\textrm{min}} & t & \\ 
\textrm{s.t.} & \displaystyle f(x) - t \le 0,  & \\
&\displaystyle g_i(x,s)\le0, & \forall s\in\Xi,\;  \forall i\in[m]  \\    
& x\in X.  &
\end{array}
\end{equation} 
Treating $t$ as $x_0$ in \eqref{opt:mehrotra-SIP}, we now verify that \eqref{opt:SIP-reform} satisfies Assumption~\ref{ass:appendix} in Appendix~\ref{appendix:convg-central-cut-surf}. Since $\Theta\times\Xi$ is bounded and $h(\cdot,\cdot)$ is continuous, $\exists C_1,C_2$ such that $h(\cdot,\cdot)\in[C_1,C_2]$ on $\Theta\times\Xi$, and the objective value of \eqref{opt:DRO} is in $[C_1,C_2]$. Also, since there is no duality gap, we can set the dual objective 
$C_1\le f(x)\le C_2$ for all $x\in X$, and set $t\in [C_1,C_2]$, without affecting the optimal solution and the optimal value.
Hence, Assumption~\ref{ass:appendix} (1) is satisfied. To verify Assumption~\ref{ass:appendix} (2), we note that
 for any $\eta>0$, we can verify that $[\overline{t},\overline{x}]$ is a Slater point
of \eqref{opt:SIP-reform}, where $\overline{t}=C_2+2\eta$ and $\overline{x}=[\theta^0, (C_2+\eta){\bf 1_m}, 0]$ ($\bf 1_m$ is the $m$-dimensional vector with all entries being 1). Now we show that Assumption~\ref{ass:appendix} (3) is also satisfied. Now let us take a subgradient of $f(x)-t$ and $g_i(x,s)$ ($\forall i\in[m]$) with respect to the decision variables $[t,x]$, and use Assumption~\ref{ass:theta_convex_h_convex}, we can set the centerality parameter $B^{\prime}$ to be $B^{\prime}:=\textrm{max}\{ \sqrt{r^2_0+1+(1/m)},\; \sqrt{B^2+L^2+1} \}$, where $L=\textrm{max}\{d(s,\xi^i):\; s\in \Xi, \; i\in[m] \}$. The oracle is assumed to be given based on the prerequisites of Algorithm~\ref{alg:cutting-surface}, hence Assumption~\ref{ass:appendix} (4) is satisfied.

We now apply Theorem~\ref{thm:convg-cut-surf-mehrotra} on the semi-infinite program \eqref{opt:SIP-reform} and the corresponding master problem \eqref{opt:master}. Part (a) directly follows from Theorem~\ref{thm:convg-cut-surf-mehrotra}~(a)-(c). By Theorem~\ref{thm:convg-cut-surf-mehrotra}~(d), for $k\ge\widehat{k}$ iterations, where $w^{(\widehat{k})}<\eta/B$, we have 
\begin{equation}
\rho^{(k)} \le 1- \frac{\eta - B^{\prime}w^{(k)}}{\eta + B^{\prime}(\overline{t} - f^*)}, \label{eqn:convg-rate}
\end{equation}
where $f^*$ is the optimal value of \eqref{opt:SIP-reform}. Since \eqref{eqn:convg-rate} holds for every $\eta>0$, 
we can select an $\eta$ to maximize $\frac{\eta - B^{\prime}w^{(k)}}{\eta + B^{\prime}(\overline{t} - f^*)}$, 
hence minimizing the upper bound of $\rho^{(k)}$ in Theorem~\ref{thm:convg-cut-surf-mehrotra}.
Let $F(\eta):= \frac{\eta - B^{\prime}w^{(k)}}{\eta + B^{\prime}(\overline{t} - f^*)}$, and substitute $\overline{t}=C_2+2\eta$ in $F(\eta)$, we have 
\begin{equation}
F(\eta)= \frac{\eta - B^{\prime}w^{(k)}}{(2B^{\prime}+1)\eta + B^{\prime}(C_2 - f^*)}, \quad F^{\prime}(\eta)=\frac{B^{\prime}w^{(k)}(2B^{\prime}+1) +B^{\prime}(C_2 - f^*) }{ [(2B^{\prime}+1)\eta + B^{\prime}(C_2 - f^*) ]^2 }.
\end{equation}
Since $w^{(k)}>0$, $f^*\le C_2$, we have $F^{\prime}(\eta)>0$ for all $\eta>0$. Therefore, the maximum value of $F(\cdot)$ is:
\begin{equation}
\underset{\eta>0}{\textrm{max}}\; F(\eta) = F(\infty) = \frac{1}{2B^{\prime}+1}.
\end{equation}
It follows that the uniform rate of convergence satisfies: $\rho \le 1- \frac{1}{2B^{\prime}+1}$.
\end{proof}

\subsection{Computational tractability of the separation problem}
\label{sec:method_sep}
We now discuss the computational difficulty of solving \eqref{opt:infeasible_detection}, which depends on the function form of $h(\theta,s)$ in $s$ for a given $\theta$, the metric $d$, and the uncertainty set $\Xi$. Since $\widetilde{v}_{m+1}\ge0$ and in most applications $d(\cdot,\cdot)$ is chosen to be a vector norm, the term $-\widetilde{v}_i-\widetilde{v}_{m+1}d(s,\xi^i)$ in \eqref{opt:infeasible_detection} is concave in $s$. Therefore, in the case where $h(\theta,s)$ is concave in $s$, and $\Xi$ is a convex set, \eqref{opt:infeasible_detection} becomes a convex optimization problem. For a very general case where $h(\theta,\cdot)$ and $d(\xi^i,\cdot)$ are continuously differentiable 
over the compact (not necessarily convex) uncertainty set $\Xi$, drawing a sufficiently many independent uniform samples $s^{t}\in\Xi$ and verifying if objective value of \eqref{opt:infeasible_detection} is greater than $\varepsilon$ can either identify a violated constraint (not necessarily the most violated constraint), or conclude that the solution of the current master problem is $\varepsilon$-optimal with high probability (See \cite{mehrotra2015} Section~5.2). Furthermore, for cases where $h(\theta,\cdot)$ is a polynomial function, $d(\cdot,\cdot)$ is an Euclidean norm, and $\Xi$ is specified by polynomial inequalities (e.g., an ellipsoid), \eqref{opt:infeasible_detection} falls in to the category of polynomial optimization. The global optimal solution in such cases can be obtained by solving a sequence of SDP relaxations (a primal approach) \cite{lasserre2001_glob-opt-polynomial} or a sequence of SOS relaxations (a dual approach) \cite{parrilo2003_sdp-relax-semialg-prob}.  
  
For some important models from statistical learning such as linear regression, linear support vector machine, logistic regression, etc.,  the loss function $h$ has the form $h(\theta_0+\theta^T x, y)$, where $x$ is the feature vector and $y$ is the response value. For this case, \eqref{opt:infeasible_detection}
can be solved efficiently using a branch-and-bound scheme based on piece-wise linear approximations of $h(\theta_0+\theta^T x, y)$. For more details about this approach, see \cite{luo2017_DRO-log-svm}.

\section{Numerical Experiments}
\label{sec:num_exp}
We investigate the following Wasserstein-robust logistic regression (WRLR) model  as a numerical example to illustrate our algorithmic ideas and the performance of \eqref{opt:DRO}:
\begin{equation}
\label{opt:WRLR}
\underset{[\theta_0,\theta]\in\Theta}{\textrm{min}}\; \underset{P\in\mathcal{P}}{\textrm{max}}\; \mathbb{E}_P \Big[ \textrm{log}\big( 1+\textrm{exp}[-y(\theta_0+\theta^Tx)] \big)  \Big],  \tag{WRLR}
\end{equation}
where $x\in\mathbb{R}^n$ is the feature vector, and $y\in\{0,1\}$ is the label. The semi-infinite reformulation of \eqref{opt:WRLR} is written as follows:
\begin{equation}
\label{opt:WRLR-SIP}
\begin{aligned}
&\underset{\theta_0,\theta,v}{\textrm{min}} \;\;  \frac{1}{m}\sum^m_{i=1} v_i + r_0\cdot v_{m+1}  \\ 
&\textrm{ s.t.}\quad  \textrm{log}\left(1 + \textrm{exp}\big[-y(\theta_0 + \theta^Tx)\big] \right) - v_i - v_{m+1}\cdot d(s,\xi^i) \le 0, \; s\in\Xi, i\in[m],  \\    
&\qquad\; [\theta_0,\theta]\in\Theta, v\in\mathcal{H},    
\end{aligned}
\end{equation}
where $s=[x,y]$, and $\xi^i=[x^i,y^i]$ is the $i$th ($i\in[m]$) observed sample. We assume that only the feature vector $x$ has uncertainty but not the label $y$.  
Therefore, the uncertainty set $\Xi$ can be written as $\Xi=\Xi_0\cup\Xi_1$, where $\Xi_0$, $\Xi_1$ are the uncertainty sets for feature vectors with the label 0, 1 respectively. The uncertainty set $\Xi_0$ is defined as an $n$-dimensional hyper-rectangle such that each dimension corresponds to an interval $[\overline{x}_j \pm \sigma_j]$ for the feature $x_j$, where $\overline{x}_j$ is the sample mean of observations with label $0$, and $\sigma_j$ is the sample standard deviation. The uncertainty set $\Xi_1$ is defined similarly. We used the $l_1$ norm to define the metric $d$ on $\Xi_0$ ($\Xi_1$), i.e., $d(x,x^{\prime})=\|x - x^{\prime} \|_1$. 

\subsection{Numerical setup}
We present computational effort in solving the semi-infinite dual problem \eqref{opt:DRO-D} of \eqref{opt:WRLR} using the cutting-surface algorithm (See Section~\ref{sec:cutting-surface-alg}).  
We also present the out of sample predictive performance of the (WRLR) model compared with the ordinary logistic regression model (LR). 
The algorithm for solving  \eqref{opt:DRO-D} were implemented in C++, and the computational tests were performed on an Intel Xeon CPU with 4 GB of RAM.  
The cutting-surface algorithm for \eqref{opt:DRO-D} consists of iteratively solving the master problem (\ref{opt:master}) and the separation problem \eqref{opt:infeasible_detection}. The convex master problem (\ref{opt:master}) is solved using Ipopt 3.12.4 \cite{ipopt} which implements a primal-dual interior point method.  The separation problem \eqref{opt:infeasible_detection} is solved using the branch-and-bound scheme based on sequentially piece-wise linear approximating $h_{\theta}(u):=\textrm{log}(1+e^{-u})$, and each convex optimization subproblem is solved using CPLEX 12.6.3.
We used 11 data sets (those with less than 60 features) from the UCI machine learning repository in our computational testing,
which are: Banknote authentication (BA), Vertebral column  (VC), Pima Indians diabetes (PID), Breast cancer Wisconsin (BCW),
Statlog heart  (ST-H), EEG eye state  (EES), SPECT heart (SPT-H), Ionosphere  (ION), SPECTF heart (SPTF-H), Spambase (SPAM), Connectionist bench (CB). 
A summary of these datasets is given in Table~\ref{table:summary_dataset}. We now describe the data generation for our test problems. We chose $m$ ($m=$\;50, 75, 100, 150) observations from each of the UCI data sets. We kept the class labels of the chosen observations unchanged.
{
 \setlength{\tabcolsep}{15pt}
 \begin{table}%[H]
 \caption{Summary of data sets from UCI Machine Learning Repository}\label{table:summary_dataset}
 \centering
  \begin{tabular}{llccc}
  \hline\hline
  Data set  & Area & No. Attrib. & No. Observ.  \\
  \hline
  BA  & Finance  & 4 & 1372 \\ 
  VC  & Health care & 6 & 310 \\
  PID & Health care & 8  & 768\\
  BCW & Health care &  9 & 699 \\
  ST-H & Health care & 13 & 270 \\
  EES & Health care & 14 & 14980 \\
  SPT-H & Health care & 22 & 267 \\
  ION & Aerospace  & 34 & 351 \\
  SPTF-H & Health care & 44 & 267 \\
  SPAM & Computer   & 57 & 4601 \\ 
  CB & Aerospace  & 60 & 208 \\
  \hline
  \end{tabular}
 \end{table}
}

\subsection{Computational effort in solving WRLR}
\label{sec:comp-perform}
The computational effort in solving the semi-infinite dual problem (\ref{opt:DRO-D}) of (WRLR) for 11 data sets is given in Table~\ref{table:comp-auc} for different choices of $m$. The numbers are averaged over the 100 experiments. Columns~4-8 give the number of outer iterations in Algorithm~\ref{alg:cutting-surface}, total number of constraints (cuts)  added to the master problem at termination, the CPU time for solving a problem instance, and the percentage of time spend in solving the master and separation problems, respectively. We also provide the CPU time for solving (LR) in Column 3.  The results show that the number of calls to the master problem is approximately $4\sim 40$ when solving (WRLR). Over these calls approximately $15m\sim50m$ cutting surfaces are added. In other words, approximately $15m\sim 50m$ artificial samples are identified. For data sets with more features, the program spends a larger fraction of time on solving master problems since their scale becomes larger. The computational time of solving (LR) models is less than 1 second for data sets with feature size less than 20, and for data sets with feature size between $30\sim60$ the  computational time is less than 20 seconds. The average time of solving (WRLR) models is $\lesssim 100$ times that of solving the (LR) models.

\subsection{Predictive performance of the WRLR model}
\label{sec:forecast-perform}
We now compare the predictive performance of the (WRLR) model with the (LR) model. For each data set, we randomly select $m$ samples out of all the observed samples to train both models.  For each combination of data set and the training sample size $m$, we performed 100 experiments. Both trained (LR) and (WRLR) models are used to predict the remaining observations from the data set, and the corresponding AUC values (area under the ROC curve) are recorded. AUC value is the most popular metric used for evaluating the performance of a model used in medical literature. In each experiment, training samples are selected randomly and independently. The mean AUC values ($\overline{\textrm{AUC}}$) over 100 experiments and the standard errors  are listed in Table~\ref{table:comp-auc}. The p-values in this table are based on the hypothesis test: $H_0: \overline{\textrm{AUC}}^{\textrm{OOS}}_{\textrm{WRLR}}\le \overline{\textrm{AUC}}^{\textrm{OOS}}_{\textrm{LR}}$ versus $H_1: \overline{\textrm{AUC}}^{\textrm{OOS}}_{\textrm{WRLR}}>\overline{\textrm{AUC}}^{\textrm{OOS}}_{\textrm{LR}}$, where $\overline{\textrm{AUC}}^{\textrm{OOS}}_{\textrm{WRLR}}$ and $\overline{\textrm{AUC}}^{\textrm{OOS}}_{\textrm{LR}}$ denote the out of sample mean AUC values corresponding to (WRLR) and (LR), respectively. Statistically, the smaller the p-value, the more likely $H_1$ is true.

To train the (WRLR) model, one needs to specify the radius $r_0$ of the Wasserstein ball.  One way to determine this radius is to use the concentration inequality $\textrm{Pr}\big(\mathcal{W}(P_{\textrm{true}}, P_0)\le r_0\big)\ge 1-\gamma$. The theoretical bounds in \cite{fournier2014_rate-convg-wass} are of limited value. Instead we used six candidate empirical Wasserstein radii: $\{0, 0.01, 0.05, 0.1, 0.5, 1\}$\footnote{Note that with $r_0=0$, the (WRLR) reduces to the (LR) model.} and a cross-validation approach to select the best $r_0$ from these. Specifically, we used the following $4$-fold cross-validation approach: we divided $m$ training samples into 4 subsets and used any three of them to train WRLR with every candidate value of $r_0$ and tested the model on the remaining subset. We finally picked the $r_0$ value corresponding to the best mean AUC  value over 4 folds. Once this $r_0$ is selected, it is used for out of sample testing on the remaining observations that is not part of the chosen $m$ samples.

The comparison shows that the quanity $(\overline{\textrm{AUC}}^{\textrm{OOS}}_{\textrm{WRLR}} - \overline{\textrm{AUC}}^{\textrm{OOS}}_{\textrm{LR}})$ ranges from -.0462 to .1005 and the relative difference ranges from -.3791 to .7122, in the studied cases. We observe that  $\overline{\textrm{AUC}}^{\textrm{OOS}}_{\textrm{WRLR}}$ is greater than $\overline{\textrm{AUC}}^{\textrm{OOS}}_{\textrm{LR}}$ in 34 ($77\%$) cases out of all 44 cases. The standard errors associated with (WRLR) are smaller than that of (LR) in 31 ($72\%$) cases. This suggests that not only the distributionally-robust model is better, its performance is also more stable. It is seen from the p-values at the significance level $\alpha=0.05$, (WRLR) outperformances (LR) in 24 ($55\%$) cases which are from
 seven data sets: BA, BCW, ST-H, SPT-H, ION, SPTF-H and SPAM.  For 7 ($16\%$) cases which are from data sets: VC, PID and EES, $\overline{\textrm{AUC}}^{\textrm{OOS}}_{\textrm{WRLR}}$ is significantly smaller than $\overline{\textrm{AUC}}^{\textrm{OOS}}_{\textrm{LR}}$, indicating (WRLR) is not as good as (LR) for these three data sets. For the remaining 13 ($29\%$) cases that are from data sets: PID, ST-H, EES, SPT-H, SPAM and CB, the difference in mean AUC is not statistically significant.

\begin{landscape}

\begin{ThreePartTable}
	\begin{TableNotes} \scriptsize  %\footnotesize
		\item[1] Number of outer iterations described in Algorithm~\ref{alg:cutting-surface}.
		\item[2] Fraction of time spent on solving master problems described in Algorithm~\ref{alg:cutting-surface}.
		\item[3] Fraction of time spent on solving separation problems described in Algorithm~\ref{alg:cutting-surface}.
		\item[4] Defined as $(\overline{\textrm{AUC}}_{\textrm{WRLR}} - \overline{\textrm{AUC}}_{\textrm{LR}})/(1- \overline{\textrm{AUC}}_{\textrm{LR}} )$.
	\end{TableNotes}
{
\scriptsize
\setlength{\tabcolsep}{4pt}
\def\arraystretch{1.2}
\begin{center}
\begin{longtable}{@{\extracolsep{0pt}}cc|cccccc|ccccccc}
\caption{Computational performance of solving WRLR and predictive performance of WRLR compared with LR  for 11 data sets. Listed values are average of 100 experiments.} \label{table:comp-auc} \\
\hline\hline
 &  & \multicolumn{6}{c|}{Computational Performance} & \multicolumn{7}{c}{Predictive Performance} \\
 \cline{3-8} \cline{9-15}  
   &  & LR & \multicolumn{5}{c}{WRLR} & \multicolumn{2}{c}{LR} & \multicolumn{2}{c}{WRLR}   \\   
 \cline{4-8}  \cline{9-9} \cline{11-12} 
 Data set & $m$ & CPU [sec] & No. main iters.\tnote{1} & No. cuts  & CPU [sec] & Master\tnote{2} ($\%$) & Sep.\tnote{3} ($\%$) & Mean AUC & S.E. & Mean AUC & S.E. & Diff.  & Rel. Diff.\tnote{4} & p-value \\
 \hline
\endfirsthead

\multicolumn{7}{c}{{\bfseries \tablename\ \thetable{} -- continued from previous page}} \\
\hline\hline
 &  & \multicolumn{6}{c|}{Computational Performance} & \multicolumn{7}{c}{Predictive Performance} \\
 \cline{3-8} \cline{9-15}  
   &  & LR & \multicolumn{5}{c}{WRLR} & \multicolumn{2}{c}{LR} & \multicolumn{2}{c}{WRLR}   \\
 \cline{4-8}  \cline{9-9} \cline{11-12} 
Data set & $m$ & CPU [sec] & No. main iters.\tnote{1} & No. cuts  & CPU [sec] & Master\tnote{2} ($\%$) & Sep.\tnote{3} ($\%$) & Mean AUC & S.E. & Mean AUC & S.E. & Diff.  & Rel. Diff.\tnote{4} & p-value \\
 \hline
\endhead

\hline
\multicolumn{9}{r}{{Continued on the next page}} \\
\hline
\endfoot

\hline\hline 
\insertTableNotes
\endlastfoot

BA	&	50	&	0.022	&	3.8	&	66.9	&	1.21	&	13.74	&	86.26	&	.9985	&	.0002	&	.9991	&	.0001	&	.0006	&	.4202	&	.0062	\\
	&	75	&	0.025	&	4.3	&	90.8	&	0.86	&	17.96	&	82.04	&	.9985	&	.0001	&	.9994	&	.0000	&	.0009	&	.5775	&	.0000	\\
	&	100	&	0.032	&	3.9	&	116.7	&	1.83	&	13.89	&	86.11	&	.9993	&	.0000	&	.9995	&	.0000	&	.0002	&	.2724	&	.0001	\\
	&	150	&	0.042	&	4.6	&	157.5	&	2.30	&	14.42	&	85.58	&	.9996	&	.0000	&	.9997	&	.0000	&	.0000	&	.1179	&	.0000	\\
	\hline																												
VC	&	50	&	0.021	&	7.2	&	195	&	1.65	&	30.51	&	69.49	&	.8782	&	.0025	&	.8320	&	.0034	&	-.0462	&	-.3791	&	1.0000	\\
	&	75	&	0.032	&	7.1	&	271.3	&	1.82	&	29.15	&	70.85	&	.8867	&	.0022	&	.8441	&	.0029	&	-.0426	&	-.3755	&	1.0000	\\
	&	100	&	0.047	&	6.6	&	323.3	&	3.73	&	32.44	&	67.56	&	.8887	&	.0023	&	.8504	&	.0025	&	-.0383	&	-.3442	&	1.0000	\\
	&	150	&	0.064	&	6.5	&	541.5	&	5.19	&	30.63	&	69.37	&	.8951	&	.0023	&	.8601	&	.0029	&	-.0350	&	-.3337	&	1.0000	\\
	\hline																												
PID	&	50	&	0.032	&	7.5	&	203.7	&	4.19	&	22.20	&	77.80	&	.7542	&	.0058	&	.7564	&	.0037	&	.0022	&	.0089	&	.3756	\\
	&	75	&	0.044	&	6.2	&	244	&	5.15	&	19.78	&	80.22	&	.8009	&	.0020	&	.8000	&	.0017	&	-.0009	&	-.0043	&	.6281	\\
	&	100	&	0.05	&	7.5	&	240.4	&	3.85	&	13.34	&	86.66	&	.7880	&	.0025	&	.7840	&	.0019	&	-.0041	&	-.0192	&	.9034	\\
	&	150	&	0.087	&	6	&	403.9	&	9.73	&	19.01	&	80.99	&	.8189	&	.0011	&	.8084	&	.0020	&	-.0105	&	-.0577	&	1.0000	\\
	\hline																												
BCW	&	50	&	0.051	&	8.6	&	251.8	&	6.19	&	31.44	&	68.56	&	.9716	&	.0020	&	.9916	&	.0004	&	.0200	&	.7040	&	.0000	\\
	&	75	&	0.076	&	9.4	&	284	&	7.88	&	27.48	&	72.52	&	.9773	&	.0016	&	.9886	&	.0010	&	.0112	&	.4954	&	.0000	\\
	&	100	&	0.097	&	8.9	&	501.4	&	12.84	&	34.28	&	65.72	&	.9790	&	.0025	&	.9940	&	.0001	&	.0150	&	.7122	&	.0000	\\
	&	150	&	0.124	&	9.6	&	786.1	&	27.31	&	41.79	&	58.21	&	.9889	&	.0007	&	.9945	&	.0001	&	.0056	&	.5049	&	.0000	\\
	\hline																												
ST-H	&	50	&	0.069	&	11.7	&	242.7	&	12.18	&	30.56	&	69.44	&	.8317	&	.0035	&	.8808	&	.0029	&	.0490	&	.2914	&	.0000	\\
	&	75	&	0.159	&	8.5	&	321.9	&	13.02	&	27.58	&	72.42	&	.8504	&	.0036	&	.8903	&	.0018	&	.0398	&	.2664	&	.0000	\\
	&	100	&	0.147	&	9.6	&	381.5	&	10.50	&	24.15	&	75.85	&	.8945	&	.0022	&	.9064	&	.0011	&	.0120	&	.1133	&	.0000	\\
	&	150	&	0.193	&	8.1	&	472.5	&	21.14	&	29.02	&	70.98	&	.8986	&	.0017	&	.8990	&	.0017	&	.0004	&	.0042	&	.4319	\\
	\hline																												
EES	&	50	&	0.041	&	15.7	&	251.9	&	7.80	&	37.72	&	62.28	&	.5874	&	.0037	&	.5902	&	.0035	&	.0029	&	.0070	&	.2877	\\
	&	75	&	0.079	&	12.1	&	435	&	11.11	&	37.37	&	62.63	&	.6095	&	.0038	&	.5996	&	.0041	&	-.0099	&	-.0252	&	.9602	\\
	&	100	&	0.322	&	11.9	&	577.8	&	17.78	&	36.55	&	63.45	&	.6221	&	.0033	&	.6175	&	.0032	&	-.0047	&	-.0124	&	.8459	\\
	&	150	&	0.383	&	8.5	&	737.8	&	31.56	&	41.78	&	58.22	&	.6245	&	.0023	&	.6139	&	.0026	&	-.0106	&	-.0282	&	.9987	\\
	\hline																												
SPT-H	&	50	&	0.241	&	21.5	&	938.8	&	38.87	&	88.91	&	11.09	&	.8126	&	.0016	&	.8176	&	.0018	&	.0050	&	.0269	&	.0203	\\
	&	75	&	0.343	&	24.5	&	1031.2	&	53.43	&	83.40	&	16.60	&	.8221	&	.0021	&	.8311	&	.0030	&	.0089	&	.0501	&	.0078	\\
	&	100	&	0.54	&	19.4	&	1122.5	&	63.08	&	83.44	&	16.56	&	.8311	&	.0022	&	.8370	&	.0037	&	.0058	&	.0346	&	.0866	\\
	&	150	&	0.987	&	13.4	&	1384	&	49.70	&	77.91	&	22.09	&	.8301	&	.0029	&	.8567	&	.0033	&	.0267	&	.1569	&	.0000	\\
	\hline																												
ION	&	50	&	0.912	&	34.7	&	740.6	&	173.37	&	63.56	&	36.44	&	.8429	&	.0024	&	.8708	&	.0021	&	.0279	&	.1775	&	.0000	\\
	&	75	&	2.415	&	28.2	&	1208.5	&	229.49	&	62.95	&	37.05	&	.8582	&	.0031	&	.8919	&	.0018	&	.0338	&	.2381	&	.0000	\\
	&	100	&	1.51	&	23.1	&	1502.5	&	441.58	&	68.51	&	31.49	&	.8606	&	.0029	&	.8967	&	.0018	&	.0360	&	.2584	&	.0000	\\
	&	150	&	2.774	&	17.8	&	1718.3	&	386.63	&	61.30	&	38.70	&	.8715	&	.0024	&	.9006	&	.0021	&	.0291	&	.2264	&	.0000	\\
	\hline																												
SPTF-H	&	50	&	2.144	&	38.9	&	438.1	&	62.68	&	70.38	&	29.62	&	.6818	&	.0048	&	.7517	&	.0056	&	.0700	&	.2198	&	.0000	\\
	&	75	&	3.718	&	31.6	&	791.4	&	112.66	&	74.41	&	25.59	&	.7255	&	.0045	&	.8260	&	.0031	&	.1005	&	.3661	&	.0000	\\
	&	100	&	3.453	&	29	&	983.7	&	144.66	&	75.50	&	24.50	&	.7638	&	.0033	&	.8368	&	.0027	&	.0730	&	.3091	&	.0000	\\
	&	150	&	5.694	&	23.7	&	1327.6	&	221.44	&	77.87	&	22.13	&	.8142	&	.0039	&	.8632	&	.0037	&	.0490	&	.2638	&	.0000	\\
	\hline																												
SPAM	&	50	&	3.655	&	50.5	&	538.2	&	266.40	&	66.10	&	33.90	&	.8969	&	.0012	&	.8984	&	.0010	&	.0015	&	.0142	&	.1796	\\
	&	75	&	4.76499	&	46.9	&	918.8	&	422.79	&	60.15	&	39.85	&	.9219	&	.0014	&	.9249	&	.0019	&	.0030	&	.0386	&	.0998	\\
	&	100	&	11.065	&	40.5	&	1368	&	409.94	&	65.66	&	34.34	&	.9164	&	.0024	&	.9331	&	.0008	&	.0166	&	.1992	&	.0000	\\
	&	150	&	15.76	&	36.8	&	1989.7	&	1414.10	&	75.31	&	24.69	&	.9346	&	.0010	&	.9387	&	.0007	&	.0041	&	.0630	&	.0005	\\
	\hline																												
CB	&	50	&	5.59802	&	44.6	&	420.4	&	120.13	&	42.30	&	57.70	&	.8081	&	.0043	&	.8136	&	.0033	&	.0055	&	.0286	&	.1549	\\
	&	75	&	9.23099	&	45.7	&	601.2	&	228.41	&	53.83	&	46.17	&	.8105	&	.0027	&	.8120	&	.0016	&	.0015	&	.0080	&	.3138	\\
	&	100	&	10.67	&	43.7	&	763.5	&	203.12	&	51.79	&	48.21	&	.8218	&	.0037	&	.8228	&	.0043	&	.0010	&	.0056	&	.4298	\\
	&	150	&	19.887	&	39.4	&	1362	&	570.95	&	60.77	&	39.23	&	.8405	&	.0061	&	.8461	&	.0062	&	.0056	&	.0352	&	.2597	

\end{longtable}
\end{center}
}
\end{ThreePartTable}

\end{landscape}

\section{Concluding Remarks}
The computational results presented in this paper used a Wasserstein-robust framework for the logistic regression model, 
where all the decision variables are continuous, and the feasible set is assumed to be convex. The modified exchange algorithm of 
this paper is applicable to a broad class of nonlinear optimization problems. These algorithms can be implemented further within a 
branch-and-bound framework. Since the decomposition framework for solving the dual of \eqref{opt:DRO} is suitable for possibly mixed-integer decision 
variables and model parameters, this framework can also be adapted to model and solve \eqref{opt:DRO} application
problems such as those arising in scheduling, logistics and supply chain management.

\section*{Acknowledgement}
We are grateful to Changhyeok Lee and Liwei Zeng for a valuable discussion. We also acknowledge NSF grants CMMI-1362003
and CMMI-1100868 that were used to support this research.

\bibliographystyle{siamplain}
\bibliography{reference-database}

\appendix

\section{Proofs and supplements for Sections~\ref{sec:reform} and \ref{sec:alg-SDP}}
\label{appendix1}
\subsection*{Proof of Proposition~\ref{lem:fP_continuity_over_P}}
\label{appendix:fcontP}
\begin{proof}
By the definition of  Wasserstein metric in (\ref{def:Kantorovich_metric}), if two probability measures $P_1$ and $P_2$ satisfying $\W(P_1,P_2)< \varepsilon$,  $\varepsilon >0$, then there exists a $K \in \mathcal{M}(\Xi\times \Xi, \mathcal{F}\times \mathcal{F})$ such that 
\begin{align}
&K(A\times\Xi) = P_1(A), \quad K(\Xi\times A) = P_2(A), \quad \forall A\in\mathcal{F}\label{eqn:prop_K_P1_P2}\\
&\int_{(s_1\times s_2)\Xi\times\Xi} d\big(s_1,s_2\big) K(ds_1\times ds_2) \le \varepsilon.  \label{eqn:prop_int_K_le_varepsilon}
\end{align}
Therefore, we have 
\begin{equation}
	\begin{aligned}
		&|f(P_1)-f(P_2)|=\left\lvert \int_{s_1\in\Xi}h(\theta,s_1)P_1(ds_1) - \int_{s_2\in\Xi}h(\theta,s_2)P_2(d s_2) \right\rvert \\
		&=\left\lvert \int_{s_1\in\Xi}h(\theta,s_1)K(ds_1\times\Xi) - \int_{s_2\in\Xi}h(\theta,s_2)K(\Xi\times ds_2)  \right\rvert \\
		&=\left\lvert \int_{(s_1\times s_2)\in\Xi\times\Xi}h(\theta,s_1)K(ds_1\times ds_2) - \int_{(s_1\times s_2)\in\Xi\times\Xi}h(\theta,s_2)K(ds_1 \times ds_2) \right\rvert \\
		&\le \int_{(s_1\times s_2)\in\Xi\times\Xi} \lvert h(\theta,s_1) - h(\theta,s_2) \rvert K(ds_1\times ds_2)  \\
		&\le C(\theta) \int_{(s_1\times s_2)\in\Xi\times\Xi} d(s_1,s_2)  K(ds_1\times ds_2) \\
		&\le C(\theta)\varepsilon,
	\end{aligned}
\end{equation}
which shows that $f$ is continuous in $(\mathcal{M}(\Xi,\mathcal{F}),\W)$.
\end{proof}

\subsection*{Proof of Proposition~\ref{prop:int_d_K_zero}}
\label{appendix:P_converge}
\begin{proof}
We divide the proof into the following two steps:
(a) Show that such a joint probability measure $K$ exists. Define a probability measure $K\in\mathcal{M}(\Xi\times\Xi,\mathcal{F}\times\mathcal{F})$ such that $K(A\times B):=P(A\cap B)$ for all $A,B\in\mathcal{F}$, and $K\left(\cup^{\infty}_{i=1} A_i\times B_i \right)=\sum^{\infty}_{i=1}K(A_i\times B_i)$ for all countable collections $\big\{ A_i\times B_i \big\}^{\infty}_{i=1}$ of pairwise disjoint sets in $\mathcal{F}\times\mathcal{F}$. It is straightforward to verify that such $K$ is as desired.

(b) Show that $\int_{(s_1\times s_2)\in\Xi\times\Xi}d\big(s_1,s_2\big) K(ds_1\times ds_2)=0$. We prove by contradiction.
 Assume $\int_{(s_1\times s_2)\in\Xi\times\Xi}d\big(s_1,s_2\big) K(ds_1\times ds_2)>\varepsilon$, for some $\varepsilon>0$,
 then we have $K(A)>0$, where $A:=\{(s_1\times s_2)\in\Xi\times\Xi:\; d(s_1,s_2)>0 \}$.
 Let $A_n:=\{(s_1\times s_2)\in\Xi\times\Xi:\; d(s_1,s_2)\ge 1/n \}$. Since $A=\cup^{\infty}_{n=1}\;A_n$,
 it follows that $\exists\;m$ such that $K(A_m)>0$. Since $A_m\in\mathcal{F}\times\mathcal{F}$, 
 it can be expressed as: $A_m=\cup^{\infty}_{i=1}S^1_i\times S^2_i$, where $S^1_i, S^2_i\in\mathcal{F}$.
 This implies that $\exists\;i$ such that $K(S^1_i\times S^2_i)>0$. By the definition of $K$, we have
 $K(S^1_i\times S^2_i)=P(S^1_i\cap S^2_i)>0$, hence $S^1_i\cap S^2_i$ is nonempty, implying that  $\exists\; s\in S^1_i\cap S^2_i$ and hence $(s\times s)\in S^1_i\times S^2_i\subseteq A_m$. However, since $d(\cdot,\cdot)$ is a metric, 
 we have $d(s,s)=0$ which contradicts to $d(s,s)\ge 1/m$.
\end{proof}

\subsection*{Duality theorem of conic linear programming}
\label{appendix:dual-thm-conic-prog}
\begin{theorem}[Proposition~2.8(iii) from \cite{shapiro2001_conic-lp}]
\label{thm:duality_of_conic_linear_optimization}
Let $V$ and $W$ be linear spaces (over real numbers), with $V^{\prime}$ and $W^{\prime}$ being their dual space respectively. Also let $C \subseteq V$ and $K\subseteq W$ be convex cones. Let $A: V \rightarrow W$ be a linear mapping and $A^*: W^{\prime}\rightarrow V^{\prime}$ be its adjoint mapping. Consider the conic linear optimization problem of the following form:
\begin{subequations}
\label{opt:CLP-gen}
\begin{align}
\emph{(P)}\quad &\underset{v\in C}{\emph{min}}  \quad \langle c, v \rangle \\
&\emph{s.t.} \quad A v + b \in K.
\end{align}
\end{subequations}
Then the dual of \emph{(P)} can be formulated as
\begin{subequations}
\begin{align}
\emph{(D)}\quad &\underset{w^*\in -K^*}{\emph{max}} \quad \langle w^*, b \rangle \\
&\emph{s.t.}\quad A^*w^* + c \in C^*,
\end{align}
\end{subequations}
where $C^*$ and $K^*$ are the dual cones of $C$ and $K$ respectively. 
\begin{enumerate}
	\item[\emph{(a)}] The weak duality holds, e.g., \emph{val(P)} $\ge$ \emph{val(D)}.
	\item[\emph{(b)}] If \emph{val(P)} is finite,  $Y$ is a finite dimensional Banach space, and the optimal solution set \emph{Sol(D)} of \emph{(D)}   is nonempty and bounded,  then \emph{val(P)}$=$\emph{val(D)}.  
\end{enumerate}
\end{theorem}

\subsection*{Convergence of the central cutting surface algorithm for convex semi-infinite programs}
\label{appendix:convg-central-cut-surf}
We summarize the convergence of the central cutting surface algorithm from \cite{mehrotra2015} for solving a general semi-infinite convex optimization problem of the form:
\begin{equation}\label{opt:mehrotra-SIP}
\begin{aligned}
\textrm{minimize}&\;\; x_0 \\
\textrm{subject to}&\;\; g(x,t)\le 0 \quad \forall t\in T, \\
 & \;\; x\in X,
\end{aligned}
\end{equation}
where  the decision variable is $x$ whose first coordinate (and also the objective) is denoted by $x_0$. 
Assume the following conditions are satisfied (following the notation from \cite{mehrotra2015}):
\begin{assumption}\label{ass:appendix}
\begin{enumerate}
\item[\emph{(1)}] The set $X\subseteq\mathbb{R}^n$ is compact and convex.
\item[\emph{(2)}] There exists a Slater point $\overline{x}$ and $\eta>0$ satisfying $\overline{x}\in X$ and $g(\overline{x},t)\le -\eta$ for every $t\in T$.
\item[\emph{(3)}] The function $f(\cdot)$ and $g(\cdot,t)$ are convex and sub-differentiable for every $t\in T$; moreover, these sub-differentials are uniformly bounded: there exists a $B>0$ such that for every $x\in X$ and $t\in T$, every subgradient $d\in \partial_x g(x,t)$ satisfies $\|d\|\le B$. 
\item[\emph{(4)}] For every point $x\in X$ that is not $\varepsilon$-feasible, there exists an oracle that can find in finite time a $t\in T$ satisfying $g(x,t)>0$. 
\end{enumerate}
\end{assumption}
The master problem of \eqref{opt:mehrotra-SIP} at the $k$th iteration has the following form:
\begin{equation}
\begin{aligned}
\textrm{maximize}&\;\; w \\
\textrm{subject to}&\;\; x_0 + w \le M^{(k-1)}, \\
 &\;\; g(x,t) + B w  \le 0 \quad \forall t\in Q^{(k-1)}, \\
 & \;\; x\in X.
\end{aligned}
\end{equation}
 \begin{theorem}{\emph{(\cite{mehrotra2015} Theorems 2-4 and Theorem 8)}}
 \label{thm:convg-cut-surf-mehrotra}
Let $(x^{(k)},w^{(k)})$ be the solution to the master problem at the $k$th iteration, and $\widetilde{x}{(k)}$ be the best know $\varepsilon$-feasible solution at the end of  the $k$th iteration. The following statements hold:
\begin{enumerate}
	\item[\emph{(a)}] If Algorithm~\emph{1} terminates in the $k$th iteration, then $\widetilde{x}^{(k-1)}$ is an $\varepsilon$-optimal solution to \eqref{opt:mehrotra-SIP}.
	\item[\emph{(b)}] If Algorithm~\emph{1} does not terminate, then there exists an index $\widehat{k}$ such that the sequence $\{ \widetilde{x}^{(\widehat{k}+i)} \}^{\infty}_{i=1}$ 
	consists entirely of $\varepsilon$-feasible solutions.
	\item[\emph{(c)}] If Algorithm~\emph{1} does not terminate, then the sequence $\{ \widetilde{x}^{(k)} \}^{\infty}_{i=1}$ has an accumulation point, and each accumulation point is an $\varepsilon$-optimal solution to \eqref{opt:mehrotra-SIP}.
	\item[\emph{(d)}]  Algorithm~\emph{1} converges linearly in objective function value between consecutive feasible cuts after the first $\hat{k}$ iterations, 
	where $\widehat{k}$ satisfies $w^{(\widehat{k})}<\eta/B$. Denote by $x^*$ an optimal solution of \eqref{opt:mehrotra-SIP}, and let $\rho^{(k)}=\frac{\widetilde{x}^{(k)}_0 - x^*_0}{\widetilde{x}^{(k-1)}_0 - x^*_0}$, then we have  
	\begin{equation}
	\rho^{(k)} \le 1-\frac{\eta-Bw^{(k)}}{\eta + B( \overline{x}_0 - x^*_0)}
	\end{equation}
\end{enumerate} 
\end{theorem}

\end{document}